\documentclass[a4paper,14pt]{amsart}
\usepackage[english]{babel}
\usepackage[utf8]{inputenc}

\usepackage{amssymb}
\usepackage{amsmath}
\usepackage{amsthm}
\usepackage{bbm}

\usepackage{enumerate}
\usepackage{tikz}
\usepackage{marginnote}
\usepackage{calligra}

\usepackage{color}

\newcommand{\hide}[1]{} 
\newcommand{\bbN}{\mathbb{N}}

\newcommand{\bbR}{\mathbb{R}}

\newcommand{\calA}{\mathcal{A}}

\newcommand{\calG}{\mathcal{G}}

\newcommand{\calK}{\mathcal{K}}
\newcommand{\calL}{\mathcal{L}}

\newcommand{\calT}{\mathcal{T}}
\newcommand{\calU}{\mathcal{U}}

\DeclareMathOperator{\id}{id}

\newcommand{\dx}{\;\mathrm{d}}

\newcommand{\sa}{{\operatorname{sa}}}

\newcommand{\norm}[1]{\left\lVert #1\right\rVert}

\def\eins{\mathbf{1}}
 
\theoremstyle{definition}
\newtheorem{definition}{Definition}[section]

\newtheorem{example}[definition]{Example}

\newtheorem{open_problem}[definition]{Open Problem}

\theoremstyle{plain}
\newtheorem{proposition}[definition]{Proposition}

\newtheorem{theorem}[definition]{Theorem}
\newtheorem{corollary}[definition]{Corollary}

\numberwithin{equation}{section}

\begin{document}

\title[Almost Interior Points and Positive Semigroups]{Almost Interior Points in Ordered Banach Spaces and the Long--Term Behaviour of Strongly Positive Operator Semigroups}

\author{Jochen Gl\"uck}
\address{Jochen Gl\"uck, Institut f\"ur Angewandte Analysis, Universit\"at Ulm, 89069 Ulm, Germany}
\email{jochen.glueck@alumni.uni-ulm.de}

\author{Martin R.\ Weber}
\address{Martin R.\ Weber, Institut f\"ur Analysis, Technische Universit\"at Dresden, 01062 Dresden, Germany}
\email{martin.weber@tu-dresden.de}

\date{\today}

\begin{abstract}
	The first part of this article is a brief survey of the properties of so-called almost interior points in ordered Banach spaces. Those vectors can be seen as a generalization of ``functions which are strictly positive almost everywhere'' on $L^p$-spaces and of ``quasi-interior points'' in Banach lattices. 
	
	In the second part we study the long--term behaviour of strongly positive operator semigroups on ordered Banach spaces; these are semigroups which, in a sense, map every non-zero positive vector to an almost interior point. Using the Jacobs--de Leeuw--Glicksberg decomposition together with the theory presented in the first part of the paper we deduce sufficiency criteria for such semigroups to converge (strongly or in operator norm) as time tends to infinity. This generalises known results for semigroups on Banach lattices as well as on normally ordered Banach spaces with unit.
\end{abstract}

\subjclass[2010]{Primary 46B40; Secondary 47D06, 47B65}

\keywords{Almost-interior points, quasi-interior points, ordered Banach space, positive operator semigroup, strict positivity, strong positivity, long-time behaviour}

\maketitle

\section{Introduction}
By an \emph{operator semigroup} -- more precisely a \emph{one-parameter operator semigroup} --
we mean a family $(T_t)_{t \in J}$ of bounded linear operators on a Banach space $X$, which satisfies 
the assumption $T_0 = \id$ and $T_{t+s} =T_t T_s$ 
for all $t,s \in J$; here, $J$ is either the set $\bbN_0 := \{0,1,2,\dots\}$ (in this case the dynamical system is modelled in discrete time) 
or the set $[0,\infty)$ (then the dynamical system is modelled in continuous time). 
In the case of continuous time it is often sensible to impose some kind of continuity assumption on the mapping $t \mapsto T_t$ which leads, 
for instance, to the theory of so-called $C_0$-semigroups.
Our two main results (Theorems~\ref{thm:strong-convergence} and~\ref{thm:uniform-convergence}) are true for both 
cases $J = \bbN_0$ and $J = [0,\infty)$; in the latter case, we do not need the assumption that the semigroup be strongly continuous 
in either of those theorems. 
However, strong continuity can sometimes be helpful to check an important assumption of Theorem~\ref{thm:strong-convergence}, 
compare Corollary~\ref{cor:perturbed-semigroup}.

It happens quite frequently that a Banach space $X$  carries an additional order structure 
which renders it a so-called \emph{ordered Banach space} (for precise definitions see Section \ref{section:ordered-banach-spaces-and-strong-positivity}).  

A bounded linear operator $T$ on an ordered Banach space $X$ is called \emph{positive} if $T$ leaves the cone of positive vectors in $X$ 
invariant, and an operator semigroup $(T_t)_{t\in J}$ on such a space 
is called \emph{positive} if it respects the order structure on $X$, i.e.\ if for each $t$ the operator $T_t$ maps positive 
vectors to positive vectors.

It is a classical insight in 
analysis that positivity of an operator semigroup $(T_t)_{t \in J}$ has far reaching consequences for its long-time behaviour; 
more precisely, when combined with other appropriate assumptions positivity is often a powerful tool to prove that $T_t$ 
converges (strongly or with respect to the operator norm) as time $t$ tends to $\infty$.

\subsection{Strong positivity assumptions and almost interior points} 
One way to obtain convergence of a semigroup is to assume that it 
improves, in a sense, the positivity of vectors it is applied to. This 
idea occurs, for instance, in \cite[Theorem~6.3]{Krein1950}, 
\cite{Makarow2000} and \cite[Theorem~4.3]{Gerlach2013}. 
Our paper uses the same approach, but in a very general setting. 

To make this notion of ``positivity improving semigroups'' or ``strongly positive semigroups'' precise one needs to distinguish several 
grades of positivity within an ordered Banach space $X$. 
If $X$ is, for instance, an $L^p$-space for some $p \in [1,\infty)$  over a $\sigma$-finite measure space, we might 
consider a function $0 \le f \in L^p$ (i.e. $f(\omega)\geq 0$ for almost all $\omega$) to be ``strongly positive'' 
if $f(\omega) > 0$ for almost all $\omega$. 
\par
Similarly, in the space of continuous real-valued functions over a given compact Hausdorff space $Q$, a function $f \ge 0$ could be considered ``strongly positive'' 
if $f(\omega) > 0$ for each $\omega \in Q$ which, by the compactness of $Q$, is equivalent to the existence of a number $\varepsilon > 0$ such that 
$f(\omega) \ge \varepsilon$ for all $\omega \in Q$. In the theory of ordered Banach spaces properties of this kind  
can be generalised to the notion of an \emph{almost interior point} of the positive cone, which is a vector $f$ in 
the positive cone such that $\langle \varphi, f\rangle > 0$ for every non-zero positive functional $\varphi$. 
This class of vectors is particularly well-studied on Banach lattices where it coincides with the class 
of \emph{quasi-interior points}; we refer for instance to \cite[Section~II.6]{Schaefer1974} for a detailed study of quasi-interior points.

In ordered Banach spaces, almost interior points have also been studied on many occasions, but a survey paper or even a book 
chapter which discusses them in detail seems to be missing. In particular, a very useful theorem of Abdelaziz and Alekhno about 
the existence of positive vectors which are \emph{not} almost interior points seems to be widely unknown. 
Thus, we use Section~\ref{section:ordered-banach-spaces-and-strong-positivity} to give a survey about almost interior points in ordered 
Banach spaces. 
Our treatment is far from being comprehensive, but we think it can 
serve as a useful source of reference for those who want to find the basic 
properties of almost interior points together with some non-trivial but essential theorems and a variety of references in one place. 
The rest of the paper, i.e.\ Sections~\ref{section:jdlg} to \ref{section:uniform-convergence-of-positive-semigroups} are dedicated 
to convergence results of operator semigroups which are ``strongly positive'' or ``positivity improving'' in the sense that the orbit 
of each non-zero positive vector under such a semigroup contains an almost interior point. 
Our main results are Theorems~\ref{thm:strong-convergence} and~\ref{thm:uniform-convergence}.

\subsection{Other approaches to the long term behaviour of positive semigroups} 
It is important to distinguish our approach from others in the literature which do not rely on strong positivity 
of the semigroup but, instead, on structural assumptions on the positive cone and on stronger regularity assumptions on the semigroups. 
Many results of this type are proved in the setting of Banach lattices and, 
in this context, convergence can for instance be proved by mainly spectral theoretic methods (see e.g.\ \cite[Chapter~C-IV]{Arendt1986}, 
\cite[Theorem~4]{Lotz1986}, \cite[Corollary~3.8]{Keicher2006}, \cite[Theorem~4.2]{Gerlach2013} and \cite{Mischler2016}) 
or by employing the structure of certain classes of positive operators 
(see e.g.\ \cite[Corollary~3.11]{Greiner1982}, \cite{Gerlach2017, GerlachConvPOS}). 

Related results are also available on more general ordered Banach sopaces, as long as the cone satisfies appropriate geometric 
assumptions; see for instance, \cite[Proposition~5 and Theorem~6]{Abdelaziz1975}, \cite{Veitsblit1985, Veitsblit1985a}, 
\cite[Corollary~2.3]{Bartoszek1992} and \cite[Part~II]{GlueckDISS} 

Besides, a somewhat special role is played by $L^1$-spaces and, 
more generally, by spaces whose norm is \emph{additive} on the positive 
cone (meaning that $\norm{u+v} = \norm{u} + \norm{v}$ for all $u,v
 \ge 0$). Such cones are said to {\it admit plastering} or to be  
{\it well-based}. Details about these spaces can, for instance, be found
 in \cite{Krasnoselskii1960}, \cite[Chapter VII]{Wulich2017} and 
\cite[Sections 3.8 and 3.9]{Jameson1970}.

For positive operator semigroups on ordered Banach spaces with additive norm very special methods for the analysis of their long-term 
behaviour are available. 
These are, for instance, stochastic approaches on $L^1$-spaces (see e.g.\ \cite[Theorems~1 and~2]{Pichor2000}, 
\cite[Theorems~1 and~2]{Kulik2015}) and, on more general spaces, results relying on so-called 
\emph{lower bounds methods} (see e.g.\ \cite[Theorem~2]{Lasota1982}, \cite[Theorem~1.1]{Ding2003}, 
\cite[Theorem~1.1 and Corollary~3.6]{GerlachLB}, \cite[Section~4]{GlueckWolffLB}) 
and results results relying on Dobrushin's ergodicity coefficient (see e.g.\ \cite{ErkursunOezcan2018} and the references therein).

Although our approach in this paper does not yield special results on spaces with additive norms, 
we find it worthwhile pointing out that those spaces are closely related, by means of duality, to spaces 
with \emph{order units} which are discussed in Section~\ref{section:ordered-banach-spaces-and-strong-positivity} 
as special cases of almost interior points.

\subsection*{Preliminaries} 
Throughout we use the following notations and conventions: we set $\bbN := \{1,2,\dots\}$ and $\bbN_0 := \bbN \cup \{0\}$. 
If $X, Y$ are  Banach spaces, then $X'$ 
denotes the \emph{dual space} of $X$ and $\calL(X;Y)$ denotes the space of bounded linear operators from $X$ to $Y$, 
which is equipped with the operator norm. We set $\calL(X):=\calL(X;X)$. 
For $x \in X$ and $x' \in X'$, the operator $x' \otimes x \in \calL(X)$ is defined by $(x' \otimes x) z = \langle x', z\rangle x$ 
for all $z \in X$.
Unless otherwise noted, the underlying scalar field of all occurring 
Banach spaces is assumed to be real. 

Let $(T_t)_{t\in J}$ be a semigroup of operators in $\calL(X)$ and $T\in \calL(X)$. 
We will say that $(T_t)_{t\in J}$ 
\begin{enumerate}[\upshape (i)]
\item \emph{strongly converges} to the operator $T$ if $\norm{T_tx-Tx}\to 0$ for any $x\in X$,  
\item \emph{uniformly converges} to the operator $T$ if $\norm{T_t -T}\to 0$,
\item \emph{converges} to the operator $T$ with respect to the weak operator topology if $|\langle T_tx-Tx, x'\rangle | \to 0$ 
for any $x\in X,\; x'\in X'$,
\end{enumerate}
in all cases, as $t$ tends to infinity.

Further notation is introduced when needed; in particular, Section~\ref{section:ordered-banach-spaces-and-strong-positivity} recalls the most 
important notions from the theory of ordered Banach spaces. 
\section{Ordered Banach spaces and almost interior points} \label{section:ordered-banach-spaces-and-strong-positivity}
\subsection{Ordered Banach spaces and duality} 

By an \emph{ordered Banach space} we mean a pair $(X,X_+)$ where $X$ is a real Banach space (whose norm we suppress in the notation $(X,X_+)$) and $X_+$ is a non-empty closed subset of $X$ such that $\alpha X_+ + \beta X_+ \subseteq X_+$ for all scalars $\alpha,\beta \ge 0$ and such that $X_+ \cap (-X_+) = \{0\}$. The set $X_+$ is called the \emph{positive cone} in $X$. In order to keep the notation as simple as possible we often simply call $X$ an ordered Banach space and thereby suppress $X_+$ in the notation. 
Although many notions and results that we mention in the following remain valid in the more general setting of \emph{ordered normed spaces} or in the setting of \emph{ordered vector spaces}, 
we restrict ourselves to ordered Banach spaces here since our main results all fit into this setting.

Let $X$ be an ordered Banach space. As is well-known, the positive cone $X_+$ induces a partial order $\le$ on $X$ which 
is given by $x \le y$ if and only if $y-x \in X_+$. 
A vector $x \in X$ is called \emph{positive} if $x \in X_+$ (equivalently, $x \ge 0$). 
Moreover, we use the notation $x < y$ for $x, y\in X$ to indicate that $x \le y$ 
but $x \not= y$. For $x,z \in X$ we call the set $[x,z] := \{y \in X: \; x \le y \le z\}$ the \emph{order interval} 
between $x$ and $z$. 

The cone $X_+$ in the ordered Banach space $X$ is called \emph{total} or \emph{spatial} if $X_+ - X_+$ is dense in $X$, and the 
cone is called \emph{generating} if $X_+ - X_+ = X$. The cone $X_+$ is called \emph{normal} if every order interval in $X$ 
is norm bounded. 

As explained in the introduction, a bounded linear operator $T: X\to Y$ between two ordered Banach space $X$ and $Y$ is 
called \emph{positive} if $TX_+ \subseteq Y_+$, and a semigroup $(T_t)_{t \in J}$ on an ordered Banach space $X$ 
(where $J = \bbN_0$ or $J = [0,\infty)$) is called \emph{positive} if the operator $T_t$ is 
positive for every time $t \in J$. 
We also note that an operator semigroup $(T_t)_{t \in J}$ on a Banach space $X$ is called \emph{bounded} 
if $\sup_{t \in J} \norm{T_t} < \infty$. 

\bigskip
Let $X'$ be the dual space of an ordered Banach space $X$. We set 
\[X'_+ := \{x' \in X': \; \langle x', x \rangle \ge 0 \text{ for all } x \in X_+\}.\] 
The elements of $X'_+$ are called the \emph{positive functionals} on $X$; we write $x' \ge 0$ to indicate that a functional $x'$ is positive. 
The set $X'_+$ is a closed convex subset of $X'$ which is invariant under multiplication by positive scalars. 
We have $X'_+ \cap (- X'_+) = \{0\}$ if and only if the cone $X_+$ in $X$ is total. In this case $X'_+$ renders the dual space $X'$ also an ordered Banach space, 
and $X'_+$ is called the \emph{dual cone} of $X_+$. 
\par
For separation of two subsets in a normed space we use the following version of the 
Hahn-Banach theorem also known as the theorem of Eidelheit (see \cite{Ei36}, \cite[Theorem~II.2.3]{Wulich2017}).

\begin{theorem}\label{separation theorem}  
Let $A$ and $B$ be non-empty convex subsets in a normed space $X$ over $\bbR$. Assume that $A$ is open and $A\cap B=\emptyset$. 
Then there exists a functional $x' \in X'$ and a real number $\gamma$ such that $\langle x', a \rangle < \gamma \le \langle x', b\rangle$ 
for all $a \in A$ and all $b \in B$.
\end{theorem}

We always have the following characterisation of positive elements in $X$, no matter whether $X_+$ is total or not:

\begin{proposition} \label{prop:positive-vectors-by-duality}
	Let $X$ be an ordered Banach space and let $x_0 \in X$. 
	Then $x_0 \in X_+$ if and only if $\langle x',x_0\rangle \ge 0$ for all $x' \in X'_+$.
\end{proposition}
\begin{proof}
	The implication ``$\Rightarrow$'' is obvious. 
	For the proof of the converse implication ``$\Leftarrow$'' assume $x_0\notin X_+$.  
	There exists an open convex neighborhood $A$ of $x_0$ such that $A\cap X_+=\emptyset$, and by applying the Eidelheit 
        theorem to the sets $A$ and $B=X_+$ we find a functional $x' \in X'$ and a real number $\alpha$ such that
	\begin{align*}
	\langle x', x_0 \rangle < \alpha \leq  \langle x',x \rangle \qquad \text{for all } x \in X_+.
	\end{align*}
	Since $0 \in X_+$, we conclude that $\alpha \le 0$, so $\langle x',x_0\rangle < 0$. 
	It only remains to show that $x'$ is positive. Let $x \in X_+$. For every $n \in \bbN$ we have $nx \in X_+$, so $n \langle x',x \rangle \ge \alpha$. 
	Dividing by $n$ and letting $n \to \infty$, we can see that $\langle x',x\rangle \ge 0$, so $x'\ge 0$.
\end{proof}

\begin{corollary} \label{cor:positive-functional-which-does-not-vanish-on-given-vector}
	Let $X$ be an ordered Banach space and let $x$ be a non-zero vector in $X$. 
	Then there exists a positive functional $x' \in X'_+$ such that $\langle x', x \rangle \not= 0$.
\end{corollary}
\begin{proof}
	If $\langle x',x\rangle = 0$ for all $x' \in X'_+$, then it follows from Proposition~\ref{prop:positive-vectors-by-duality} 
	that $x \ge 0$ and $-x \ge 0$, so $x = 0$.
\end{proof}

\subsection{Almost interior points and related concepts}

The following definitions of several types of special points in ordered Banach spaces are essential for our further investigations.  

For a vector $u\geq 0$ in an ordered Banach space $X$ the set 
\begin{align*}
	X_u  = \{y \in X: \, \exists \lambda \in [0,\infty) \text{ such that } \pm y \le \lambda u\}
\end{align*} 
is a vector subspace which is called the \emph{subspace} (in $X$) \emph{of bounded elements with respect to $u$} or the \emph{principal ideal generated by $u$}. 
It is clear that $X_u=\bigcup_{\lambda \in [0,\infty)} [-\lambda u,\lambda u]$. 
\begin{definition} \label{def:versions-of-interior-points}
	Let $X$ be an ordered Banach space.
	\begin{enumerate}
		\item[(i)] A functional $x' \in X'$ is called \emph{strictly positive} if $\langle x',x\rangle > 0$ for every $0 < x \in X$.
		\item[(ii)] A vector $x \in X$ is called an \emph{almost interior point\footnote{\,Unfortunately, the terminology is by far 
		not unanimous: in \cite{Krasnoselskii1989} such a point is called quasi-interior!} of $X_+$} if $\langle x',x\rangle > 0$ for every non-zero positive functional $x' \in X'$. 
		
		\item[(iii)] A vector $x \in X$ is called an \emph{order unit} if for each $y\in X$ there exists a real number $\varepsilon>0$ such that $x\geq \varepsilon y$.   	

		\item[(iv)] A vector $x\in X$ is called a \emph{quasi-interior point of $X_+$} if $x\geq 0$ and the set $ X_x$ is dense in $X$.

		\item[(v)] A vector $x \in X$ is called an \emph{interior point} of $X_+$ if it is contained in the topological interior ${\rm int}(X_+)$ of $X_+$. 
	\end{enumerate}
\end{definition}
We note that almost interior points of $X_+$ and order units always belong to $X_+$, and that every strictly positive functional is positive.  
If for some $u$ the subspace $X_u$ coincides with $X$ then $u$ is an order unit, 
see Proposition \ref{prop:characterisation-of-order-units}. 
A vector $x\in X$ is an almost interior point of $X_+$ if (and only if) $x'\in X'_+$ and $\langle x',x\rangle=0$ 
implies $x'=0$. 
Moreover, every quasi-interior point $x$ of $X_+$ is an almost interior point of $X_+$. 
  Indeed, let $x \in X_+$ be a quasi-interior point and let $x' \in X'_+$ such that $\langle x',x\rangle = 0$. 
Then $\langle x',y\rangle = 0$ for all $y \in [0,x]$ and therefore, $\langle x',z\rangle = 0$ for all $z \in [-x,x]$, since each such $z$ can be written as
\begin{align*}
	z = \frac{x+z}{2} - \frac{x-z}{2} \in [0,x] - [0,x].
\end{align*}
Hence, $x'$ vanishes on the dense subset $X_x$ 
and thus on the entire space $X$, so $x' = 0$. 

One might ask whether, conversely, every almost interior point of $X_+$ is also a quasi-interior point of $X_+$. Schaefer remarks  
in \cite[p.\,136, assertion~(K)]{Schaefer1960} that this is not true, in general; he attributes this observation to Klee. 
Another counterexample can be found in \cite[Section~3.6]{Krasnoselskii1989}. 
However, in both counterexamples the positive cone is only total, but not generating. 
Thus, the following questions seem still to be open:

\begin{open_problem} \label{openproblem:almost-vs-quasi-interior-points}
	Let $(X,X_+)$ be an ordered Banach space with generating cone.
	\begin{enumerate}[\upshape (a)]
		\item Is every almost interior point of $X_+$ also a quasi-interior point of $X_+$?
		\item In case that the answer to question~(a) is negative, does it become positive if we assume, in addition, that the positive cone is normal? 
	\end{enumerate}
\end{open_problem}

Corollary~\ref{cor:almost-interior-points-and-interior-points} below shows that the answer to question~(a) above is affirmative in case that 
the positive cone $X_+$ has non-empty topological interior, i.e. ${\rm int}(X_+)\neq \emptyset$. 
Moreover, it is well-known (and not difficult to show by means of a quotient space argument) that the answer to question~(a) is 
affirmative in case that $(X,X_+)$ is a Banach lattice (see e.g.\ \cite[Theorem~II.6.3]{Schaefer1974} for details).

The following characterisation of almost interior points also gives an interesting perspective on the relation between 
almost interior and quasi-interior points; it can be found in \cite[Theorem~1.2]{Bakhtin1968}: 
a point $x\in X_+$ is almost interior if and only if $\overline{X_++\{\bbR x\}}=X$, i.e.\ if and only if the union of 
the ``semi-bounded'' intervals $\{y \in X: \, y \ge -\lambda x\}$ for $\lambda \in [0,\infty)$ is dense in $X$.

From an operator theoretic perspective, almost interior points seem to be much more accessible than quasi-interior points; 
since the former are also the more general concept,
we focus on almost interior points throughout the article. 

The set of all almost interior points supplemented by the zero vector is a cone (more precisely, a subcone of $X_+$) in $X$. 
If the set of almost interior points is not empty then it is dense in $X_+$. 
Indeed, if $u$ is an almost interior point and $x\in X_+$ then $x+\varepsilon u$ 
is an almost interior point for all $\varepsilon > 0$.

The existence of almost interior points plays a remarkable role in solving some questions on the extension of 
linear functionals to positive ones, see Theorem \ref{thm:extending-functionals-via-almost-interior-points}. 
Let us recall the following sufficient criteria for the existence of quasi-interior 
(and, according to what was mentioned just before, also almost interior) points and strictly positive functionals; 
see \cite{Bakhtin1968}, \cite{Krein1950} and \cite[Kapitel~II]{Wulich2017}. 
For the convenience of the reader we include the proofs, where in (b) we prove even the existence of a quasi-interior point - 
a little more than in the original paper \cite{Bakhtin1968}.
\begin{theorem}\label{existence-theorem}
	Let $X$ be an ordered Banach space and assume that $X$ is separable.
	\begin{enumerate}[\upshape (a)]
		\item There exists a strictly positive functional in $X'$.
		\item If the positive cone $X_+$ is total, then there exists a quasi-interior point of $X_+$.
	\end{enumerate}
\end{theorem}
\begin{proof}
	We may assume that $X \not= \{0\}$.
		
	(a) Let $B'$ denote the closed unit ball in $X'$ and endow $B'$ with the weak${}^*$-topology. 
	The separability of $X$ implies that $B'$ is metrizable and separable. 
        Hence, the subset $B' \cap X'_+$ of $B'$ is also separable 
	with respect to the weak${}^*$-topology. 
	Let $(x_n')_{n \in \bbN}$ be a weak${}^*$-dense  sequence in $B' \cap X'_+$. 
	Define
	\begin{align*}
		x' := \sum_{n \in \bbN} \frac{1}{2^n}x'_n.
	\end{align*}
	Then $x'$ is a positive functional on $X$. We are going to show that $x'$ is strictly positive. To this end, let $x \in X_+$ be such that 
	$\langle x',x\rangle = 0$. 
        Since $x$ is positive, this implies that $\langle x_n', x\rangle = 0$ for all $n \in \bbN$, thus $\langle x',x\rangle = 0$ for all $x' \in B' \cap X'_+$ and hence, 
        even $\langle x',x\rangle = 0$ for all $x' \in X'_+$. 
        According to Corollary~\ref{cor:positive-functional-which-does-not-vanish-on-given-vector} this implies that $x = 0$, 
        so $x'$ is indeed strictly positive.
	
	(b)\footnote{\, For a slightly more general type of ordered topological vector spaces, assertion (b) can be found in \cite[paragraph V.7.6]{Schaefer1971}.} 
	Since $X$ is separable as a metric space, so is its subset $X_+ \setminus \{0\}$. Hence, there exists a sequence 
	$(x_n)_{n \in \bbN}, \, x_n>0$, which is dense in $X_+$. 
	We define
	\begin{align*}
		x := \sum_{n \in \bbN} \frac{1}{2^n\norm{x_n}}x_n.
	\end{align*}
	Let $X_x$ be as defined at the beginning of Subsection~2.2.  
	We first prove that the closure $\overline{X_x}$ of $X_x$ contains the positive cone $X_+$; to this end, let $y \in X_+$ and let $\varepsilon > 0$. 
	Then there exists an index $k \in \bbN$ such that $\norm{x_k - y} < \varepsilon$. On the other hand we have
	\begin{align*}
		0 \le x_k \le \norm{x_k}2^k \sum_{n \in \bbN} \frac{x_n}{\norm{x_n} 2^n} = \norm{x_k}2^k\, x,
	\end{align*}
	so $x_k \in \left[0,\norm{x_k}2^kx\right] \subseteq \left[-\norm{x_k}2^kx, \norm{x_k}2^kx\right] \subseteq X_x$. 
	This proves that $y \in \overline{X_x}$ and thus $X_+ \subseteq \overline{X_x}$. Since $X_x$ is readily seen to be a vector subspace of $X$, so is its closure. 
	We conclude that $X_+ - X_+ \subseteq \overline{X_x}$ and hence, due to the totality of $X_+$ one has  
	$X = \overline{X_+ - X_+} \subseteq \overline{X_x} \subseteq X$. 
	Therefore, $x$ is indeed a quasi-interior point of $X_+$.
\end{proof}

The following geometric characterisation of almost interior points is very useful. 
Let $C$ be a closed convex set in a real Banach space $X$. 
A vector $c \in C$ is called a \emph{support point} of $C$ if there exists a 
non-zero functional $x' \in X'$ such that $\langle x',c \rangle \ge \langle x',x\rangle$ 
for all $x \in C$, i.e., $x'$ attains its maximum on $C$ at the point $c$ (see \cite{Bishop1963}). 
Using this terminology, we can characterise those positive vectors in an ordered Banach space which are \emph{not} almost interior points 
of the positive cone:

\begin{proposition} \label{prop:almost-interior-points-via-support-functionals}
	Let $X$ be an ordered Banach space and let $x \in X_+$. The following assertions are equivalent:
	\begin{enumerate}[\upshape (i)]
		\item $x$ is not an almost interior point of $X_+$.
		\item $x$ is a support point of $X_+$.
	\end{enumerate}
\end{proposition}
\begin{proof}
	``(i) $\Rightarrow$ (ii)'' If $x$ is not an almost interior point of $X_+$, then we can find a non-zero positive 
	functional $x' \in X'$ such that $\langle x',x\rangle = 0$. Hence the functional $-x'$ satisfies  
	\begin{align*}
		 \langle -x',y\rangle \le 0 = \langle -x',x\rangle 
	\end{align*}
	for all $y \in X_+$. So, $x$ is a support point of $X_+$.
	
	``(ii) $\Rightarrow$ (i)'' If $x$ is a support point of $X_+$, then there exists a non-zero functional $x' \in X'$ 
such that $\langle x',x\rangle \ge \langle x',y\rangle$ for all $y \in X_+$. 
Note that this implies $\langle x',x\rangle \ge 0$ since $0 \in X_+$. 
Moreover, for any $y \in X_+$ and $n\in \bbN$ we also have $ny \in X_+$ and thus
	\begin{align*}
		\frac{1}{n} \langle x',x\rangle \ge \langle x',y\rangle,
	\end{align*}
	for all $n\in \bbN$. Hence, $0\geq\langle x',y \rangle$. As $x$ is itself an element of $X_+$, we conclude 
	that $0 \le \langle x',x\rangle \le 0$. So $\langle x',x\rangle = 0$. 
	Thus, $-x'$ is a positive non-zero functional which maps $x$ to $0$. So $x$ is not an almost interior point of $X_+$.
\end{proof}

Let $C$ be a closed convex set in a Banach space $X$ and suppose that $C$ has non-empty interior. 
Then it is not difficult to see that $C$ coincides with the closure of its interior ${\rm int}(C)$. 
Thus, it is an easy consequence of the Hahn--Banach separation theorem that a vector $x \in C$ is a support point of $C$ 
if and only if it is contained in the topological boundary of $C$. 
Thus, we  obtain the subsequent corollary as a consequence of Proposition~\ref{prop:almost-interior-points-via-support-functionals}. 
\begin{corollary} \label{cor:almost-interior-points-and-interior-points}
	Let $X$ be an ordered Banach space and assume that the positive cone $X_+$ has non-empty interior. 
	Then the following assertions are equivalent for each vector $x \in X_+$:
	\begin{enumerate}[\upshape (i)]
		\item $x$ is an almost interior point of $X_+$.
		\item $x$ is a quasi-interior point of $X_+$.
		\item $x$ is an interior point of $X_+$.
	\end{enumerate}
\end{corollary}
\begin{proof}
	``(iii) $\Rightarrow$ (ii)'' Let $y \in X$. Due to assertion~(iii) the positive cone contains an open neighbourhood of $x$, 
	so there exists a number $\delta > 0$  such that $x - \delta y$ and $x + \delta y$ are both contained in $X_+$. 
	This shows that $y \le x/\delta$ and $y \ge - x/\delta$, so $y$ is contained in the order 
                interval $[-\frac{1}{\delta}x, \frac{1}{\delta}x]$. 
	Thus it is proved that $X_x = X$; in particular, $x$ is a quasi-interior point of $X$.  
	
	``(ii) $\Rightarrow$ (i)'' As shown after Definition~\ref{def:versions-of-interior-points}, every quasi-interior point of $X_+$ 
	is also an almost interior point of $X_+$.
	
	``(i) $\Rightarrow$ (iii)''	If $x$ is not contained in ${\rm int}(X_+)$, then it is contained in the boundary of $X_+$. 
	Hence, by what has been said right before the corollary, $x$ is a support point of $X_+$ and therefore, 
	Proposition~\ref{prop:almost-interior-points-via-support-functionals} implies that $x$ is not an almost interior point of $X_+$.
\end{proof}

The positive cone of an ordered Banach space can contain an almost interior point only if the cone is total:

\begin{proposition} \label{prop:almost-interior-points-imply-total-cone}
	Let $X$ be an ordered Banach space and assume that there exists an almost interior point $x$ of $X_+$. 
	Then the positive cone $X_+$ is total.
\end{proposition}
\begin{proof}
	Let $F$ denote the closure of $X_+ - X_+$ and assume that $F \not= X$. Since $F$ is a closed vector subspace of $X$ which is not equal to $X$,  
	by means of the separation Theorem  \ref{separation theorem}, there exists a non-zero functional $x' \in X'$ which vanishes on $F$; 
	in particular, $x'$ vanishes on $X_+$. Hence, $x'$ is a non-zero positive functional, but $\langle x', x\rangle = 0$ for each $x \in X_+$. 
	This shows that no vector $x \in X_+$ is an almost interior point of $X_+$.
\end{proof}

A deep result of Bishop and Phelps \cite[Theorem~1]{Bishop1963} says that, if $C$ is a closed convex set in a real Banach space $X$, 
then its support points are dense in its boundary. This is not difficult to show in case that $C$ has non-empty interior (in fact, 
in this case the set of support points of $C$ coincides with the boundary of $C$), but it is also true for sets with empty interior. 
Moreover, we point out that mentioned theorem of Bishop and Phelps does not assume the set $C$ to be bounded (in contrast to a, presumably better known, 
result of the same authors which asserts that the so-called \emph{support functionals} of $C$ are dense in the dual space $X'$ in case that 
$C$ is bounded \cite[Corollary~4]{Bishop1963}). 
If we combine the result of Bishop and Phelps with Proposition~\ref{prop:almost-interior-points-via-support-functionals}, we obtain the 
following result on the existence of almost interior points and strictly positive functionals. This will be important in   
 the proof of our main result in Theorem \ref{thm:strong-convergence}.

\begin{theorem} \label{thm:existence-of-non-almost-interior-points}
	Let $X$ be an ordered Banach space of dimension at least $2$ and assume that $X_+ \not= \{0\}$. 
	Then there exist a non-zero point $x \in X_+$ which is not 
	an almost interior point of $X_+$ and a non-zero functional $x' \in X'_+$ which is not strictly  positive.
\end{theorem}
\begin{proof}
	By the result of Bishop and Phelps quoted right before this theorem, the set of support points of $X_+$ is dense in the boundary of $X_+$. 
	As $\dim X \ge 2$ and as the cone $X_+$ is neither equal to $\{0\}$ nor to $X$, it follows that $X_+$ has a non-zero boundary point $y$. 
	Hence, due to the density of the support points of $X_+$ in the boundary of $X_+$, there exists a support point $x$ 
        of $X_+$ with $\norm{x-y}  \leq \frac{1}{2}\norm{y}$, so $x \not= 0$. 
	According to Proposition~\ref{prop:almost-interior-points-via-support-functionals}, $x$ is not an almost interior point of $X_+$. 
        \par
	On the other hand, if every non-zero positive functional in $X'$ was strictly positive, then every non-zero vector in $X_+$ would 
	be an almost interior point of $X_+$. Hence, there exists a non-zero positive functional in $X'$ which is not strictly positive.
\end{proof}

The result in Theorem~\ref{thm:existence-of-non-almost-interior-points} has a somewhat curious history:
in \cite[Remark~(iii) on p.\,270]{Schaefer1971} Schaefer discusses the relation of the assertion of 
Theorem~\ref{thm:existence-of-non-almost-interior-points} to a spectral theoretic question; the discussion implies 
that Schaefer considered the question whether the assertion of Theorem~\ref{thm:existence-of-non-almost-interior-points} is true or not to be open. 
A few years later, Abdelaziz used the argument presented in the proof of Theorem~\ref{thm:existence-of-non-almost-interior-points} above 
in the paper \cite{Abdelaziz1975}. 
Yet, he did not seem to be aware of the open problem of Schaefer, which was solved by his argument and moreover, 
the assertion of Theorem~\ref{thm:existence-of-non-almost-interior-points} is not easy to find in \cite{Abdelaziz1975} for two reasons: 
(i) the result is only contained implicitly in the  proof of \cite[Theorem~4]{Abdelaziz1975} and 
(ii) this theorem is only stated for the special case of ordered Banach spaces which have a normal and generating cone and which admit, in addition, 
the Riesz-decomposition property -- however, 
the very part of the proof of \cite[Theorem~4]{Abdelaziz1975} which we are referring to works perfectly fine for arbitrary ordered Banach spaces. 
The first explicit statements of Theorem~2.10 in the literature which we are aware of are much younger and due to Katsikis and Polyrakis \cite[Proposition 7]{KP06} and Alekhno \cite[Theorem 6]{Alekhno2013}. Katsikis and Polyrakis state the theorem only for quasi-interior points, but their proof -- which also uses the Bishop--Phelps theorem -- works for almost interior points, too. Alekhno states the theorem for almost interior points (but he uses a different terminology and calls them quasi-interior). He refers to the above mentioned open problem of Schaefer and to the paper of Katsikis and Polyrakis, but apparently was not aware of Abdelaziz' argument.
Alekhno's proof employs the so-called \emph{drop theorem} which originally goes back to Danes \cite{Danes1972} 
(see also \cite[Corollary~7]{Brezis1976} 
for a rather abstract approach to this result) and which can be seen as a non-linear version of the theorem of Bishop and Phelps that 
we used in the above proof. 

Before we conclude this section with a few examples, we recall a further characterisation of interior points of a positive 
cone which is, in contrast to Corollary~\ref{cor:almost-interior-points-and-interior-points},  applicable even without the assumption 
that the cone has non-empty interior. 
This characterisation is often helpful to identify interior elements of the cone (if any of them exist) in applications.
\par 
\begin{proposition} \label{prop:characterisation-of-order-units}
	Let $X$ be an ordered Banach space and let $u \in X$. The following assertions are equivalent:
	\begin{enumerate}[\upshape (i)]
		\item The vector $u$ is an order unit of $X$.
		\item The positive cone $X_+$ is generating and for every $x \in X_+$ there exists a real number $\varepsilon > 0$ 
                      such that $u \ge \varepsilon x$.
                \item The vector $u$ is positive and the principal ideal $X_u$ is equal to $X$. 		
                \item There exists a real number $\varepsilon > 0$ such that $u \ge x$ for all vectors $x \in X$ of norm $\norm{x} \le \varepsilon$.
		\item The vector $u$ is an interior point of $X_+$.
	\end{enumerate}
\end{proposition}
\begin{proof} We show ``(i) $\Rightarrow$ (ii) 
$\Rightarrow$ (iv) $\Rightarrow$ (v) $\Rightarrow$ (iv) $\Rightarrow$ 
(i)'' and ``(ii) $\Rightarrow$ (iii) $\Rightarrow$ (i)''.

	``(i) $\Rightarrow$ (ii)'' Let $u$ be an order unit. Then the second part of assertion~(ii) is obviously true. 
In order to prove that $X_+$ is generating, let $x \in X$. 
Since $u$ is an order unit we have $\frac{1}{\varepsilon} u \ge x$ for some $\varepsilon > 0$. 
As $x = \frac{1}{\varepsilon} u  - (\frac{1}{\varepsilon}u - x)$, the vector $x$ is the difference of two positive vectors.
	
	``(ii) $\Rightarrow$ (iv)'' We first show that there exists a real number $\hat \varepsilon > 0$ such that $u \ge x$ 
for all $x \in X_+$ of norm at most $\hat \varepsilon$. 
         Assume the contrary. For each integer $n \in \bbN$ we can then find a vector $x_n \in X_+$ which has norm at 
         most $\frac{1}{n^3}$ and satisfies $u \not\ge x_n$. 
        Define
	\begin{align*}
		x := \sum_{n \in \bbN} nx_n,
	\end{align*}
	where the series converges absolutely in $X$. Note that the vector $x$ is positive and that $\frac{1}{n}x$ 
        dominates $x_n$ for each index $n$. Hence, $u \not\ge \frac{1}{n} x$ for any $n \in \bbN$. 
        This contradicts~(ii). 
        Thus, there exists a real number $\hat \varepsilon > 0$ with the property claimed above.\\
Now we use that a closed generating cone in an ordered Banach space is always \emph{non-flatted}, i.e.  
we find a number $M > 0$ such that each vector $x \in X$ can be 
	written as $x = x_1 - x_2$ for two positive vectors $x_1,x_2$ of  norm at  most $M\norm{x}$ 
(see for instance \cite[Theorem~2.37]{Aliprantis2007} or \cite[Proposition~1.1.2]{Batty1984}).  
	We set $\varepsilon := \frac{\hat \varepsilon}{M}$. 
	Let $x \in X$ be a vector of norm $\norm{x} \le \varepsilon$. 
We decompose $x$ as $x = x_1 - x_2$,  where $x_1$ and $x_2$ are positive vectors whose norms satisfy  
        $\norm{x_1}, \norm{x_2} \le M \norm{x} \le M \varepsilon = \hat \varepsilon$. 
Then $x \le x_1 \le u$ by what we have shown above.
	
	``(iv) $\Rightarrow$ (v)'' Let $\varepsilon > 0$ be as in~(iv) and let $y \in X$ be a vector which satisfies 
$\norm{u-y} \le \varepsilon$. 
Then $u - y \le u$ by assertion~(iv), so $y \ge 0$. 
This shows that the (closed) ball with center $u$ and radius $\varepsilon$ is contained in $X_+$, so $u$ is an interior 
point of $X_+$.
	
	``(v) $\Rightarrow$ (iv)'' If $u$ is an interior point of $X_+$, then there exists a number $\varepsilon > 0$ such 
that the closed ball with center $u$ and radius $\varepsilon$ is contained in $X_+$. 
Now, let $x \in X$ be a vector of norm at most $\varepsilon$. 
          Then $\norm{u - (u-x)} = \norm{x} \le \varepsilon$ shows that $u-x$ lies in the closed ball centered at 
$u$ with radius $\varepsilon$, so $u-x$ is positive and thus, $x \le u$.
	
	``(iv) $\Rightarrow$ (i)'' This implication is obvious.

        ``(ii) $\Rightarrow$ (iii)'' Clearly, (ii) implies that $u \ge 0$ and that $X_+ \subseteq X_u$. As the 
positive cone is generating, we conclude that even $X \subseteq X_u$. 

        ``(iii) $\Rightarrow$ (i)'' This implication is obvious.
\end{proof}

\subsection{Examples}

In this subsection we give a few examples to illustrate the concepts discussed above.

\begin{example}[$L^p$-spaces] \label{ex:Lp-spaces} 
	Let $p \in [1,\infty)$ and let $(\Omega,\mu)$ be a $\sigma$-finite measure space; we endow the Banach space $L^p(\Omega,\mu)$ with its canonical order. 
	Then a vector $f \in L^p(\Omega,\mu)$ is an almost interior point of $L^p(\Omega,\mu)_+$ if and only if $f$ is a quasi-interior point 
	of $L^p(\Omega,\mu)_+$ (as $L^p(\Omega,\mu)$ is a Banach lattice) and so, if and only if $f(\omega) > 0$ for almost all $\omega \in \Omega$. 
	The positive cone in $L^p(\Omega,\mu)$ is generating, but it has empty interior unless $L^p(\Omega,\mu)$ is finite dimensional.
				
		On the other hand, the positive cone in $L^\infty(\Omega,\mu)$ always has non-empty interior. 
		A vector $f \in L^\infty(\Omega,\mu)$ is an interior point of the positive cone (equivalently: an almost or quasi-interior point, see Corollary~\ref{cor:almost-interior-points-and-interior-points}) if and only if $f(\omega) \ge \varepsilon$ for a number $\varepsilon > 0$ and for almost all $\omega \in \Omega$.
\end{example}

\begin{example}[Spaces without almost interior points] \leavevmode
	\begin{enumerate} 
	[\upshape (a)] 	
\item Let $T$ be  an uncountable set and let $p \in [1,\infty)$. If we endow the 
        space $\ell^p(T)$ with its pointwise order 
		(which renders it a Banach lattice), then it is not difficult to see that the positive cone does not contain any almost 
		interior points at all (see for instance \cite[Section~II.8]{Wulich2017})
\item A slightly more exotic example of a Banach lattice without almost interior 
        points is the following: let $C_b(\bbR)$ denote the space 
		of all bounded and continuous real-valued functions on $\bbR$ and endow the space
		\begin{align*}
			X := C_b(\bbR) \cap L^2(\bbR)
		\end{align*}
		with the pointwise order and the norm given by $\norm{f} := \norm{f}_\infty + \norm{f}_2$ for all $f \in X$, 
		where $\norm{f}_\infty$ denotes the {\rm sup}-norm in $C_b(\bbR)$. Then $X$ is a Banach lattice.
		
		We show that the positive cone $X_+$ does not contain any almost interior point. Fix $f \in X_+$. 
		Then there exists a sequence $(\omega_n)_{n \in \bbN}$ in $\bbR$ such that $\lvert\omega_n\rvert \to \infty$ and $f(\omega_n) \to 0$ as $n \to \infty$. 
		Moreover, we can find a function $g \in X$ which takes the value $1$ at each point $\omega_n$.
		
		Now, let $\calU$ be a free ultra-filter on $\bbN$ and define a functional 
		$\varphi \in X'$ by $\langle \varphi, h\rangle := \lim_{n \to \calU} h(\omega_n)$ 
		for each $h \in X$. Then $\varphi$ is a positive functional, and it is not zero since it maps $g$ to the number $1$. 
		However, $\langle \varphi, f\rangle = 0$, which shows that $f$ is not an almost interior point of $X_+$.
	\end{enumerate}
\end{example}		

\begin{example}[Spaces of continuous functions] \label{ex:spaces-of-continuous-functions} 
	Let $Q$ be a locally compact Hausdorff space and let $C_0(Q)$ denote the space of all real-valued 
	continuous functions on $Q$ that 
	vanish at infinity (endowed with the supremum norm). A function $f \in C_0(Q)$ is an almost interior point 
	(equivalently: a quasi-interior point) 	of the positive cone if and only if $f(\omega) > 0$ for all $\omega \in Q$.
		
		The positive cone in $C_0(Q)$ is always generating (as $C_0(Q)$ is a Banach lattice); it has non-empty interior if and only if 
		$Q$ is compact, and in this case, a function $f$ is an (almost) interior point of the positive cone if and only 
		if $f(\omega) > 0$ for all $\omega \in Q$ if and only if $f(\omega) \ge \varepsilon$ for some $\varepsilon > 0$ and all 
		$\omega \in Q$.
\end{example}
The following example requires some more detailed knowledge on $C^*$-algebras and is addressed to readers who are interested in this special 
field. 
\begin{example}[$C^*$-algebras] \label{ex:c-star-algebras}
	Let $\calA$ be a non-zero $C^*$-algebra. 
	Recall that $\calA$ is called \emph{unital} if it contains an element 
	$\eins_\calA$ which is neutral with repect to multiplication. We will tacitly use  some facts of the spectral theory in 
	(unital and non-unital) $C^*$-algebras and refer to the standard literature (see for instance \cite{Murphy1990}) 
	for more information.
	
	The self-adjoint part $\calA_\sa:=\{a\in \calA\colon a^*=a\}$ is a real Banach space, and it becomes an ordered Banach space if we 
	endow it with its usual cone $\calA_\sa^+ := \{a \in \calA_\sa: \, \sigma(a) \subseteq [0,\infty)\}$, where $\sigma(a)$ denotes 
	the \emph{spectrum} of $a$.  
	The ordered Banach space $(\calA_\sa, \calA_\sa^+)$ is not lattice ordered unless $\calA$ is commutative (this is a classical result of 
	Sherman \cite[Theorems~1 and~2]{Sherman1951}, see \cite{GlueckCStar} for a new approach to this result based on the spectral theory of 
	positive operators). 
	\par
	Now we describe the almost interior and interior points of $\calA_\sa^+$:
	\begin{enumerate}
		\item[(i)] \textit{An element $a \in \calA_\sa^+$ is an almost interior point of $\calA_\sa^+$ if and only if it is a quasi-interior point of $\calA_\sa^+$ if and only if the set $a\calA a$ is dense in $\calA$.}
	\begin{proof} The 
fact that $a$ is an almost interior point if and only if $a\calA a$ is 
dense in $\calA$ follows from the theory of hereditary $C^*$-algebras; 
see \cite[Exercise~5(c) on p.\,108 and  Remark~5.3.1]{Murphy1990} for details 
and see for instance 
\cite[Section~3.2]{Murphy1990} for more information on hereditary $C^*$-algebras. 
It remains to show that every almost interior point is 
quasi-interior\footnote{\, We were kindly informed about this result,
 as well as its proof, by one of the referees.}, so let $a$ be an 
almost interior point of $\calA_\sa^+$. For each $n \in \bbN$ define two
 continuous and bounded functions $g_n,f_n: [0,\infty) \to [0,\infty)$ by \begin{align*}
  g_n(t) = \min\{nt^{1/2}, t^{-1/2}\} 
\quad \text{and} \quad f_n(t) = t^{1/2}g_n(t) = \min\{nt, 1\} 
 \end{align*} 
for $t \in [0,\infty)$. Note that all functions $f_n$ and $g_n$ vanish at 
$0$, so $f_n(a)$ and $g_n(a)$ are elements of $\calA$, even if $\calA$ is not 
unital. For $t> 0$ the values $f_n(t)$ increase to $1$ as $n \to 
\infty$, so it follows from \cite[proof of 
Proposition~3.10.5]{Pedersen1979} that $(u_n) := (f_n(a))$ is an 
approximate unit in $\calA$. For each $x \in \calA_\sa$ this implies 
that $u_n x u_n \to x$ as $n \to \infty$ (since $u_n x u_n - x = (u_n x -
 x)u_n +x u_n - x$ and since $(u_n)$ is norm bounded). On the other 
hand, each element $u_n x u_n$ is contained in the principal ideal 
generated by $a$ within $\calA_\sa$. Indeed, we have 
\begin{align*} 
 u_n x u_n = a^{1/2} g_n(a) x g_n(a) 
a^{1/2} \le \norm{g_n(a) x g_n(a)} a \le \norm{g_n}_\infty^2 \norm{x} a 
\end{align*} 
and, by replacing $x$ with $-x$, also $u_n x u_n \ge - \norm{g_n}_\infty^2 
\norm{x} a$. This proves that $a$ is a quasi-interior point of 
$\calA_\sa^+$. 
\end{proof}
\item[(ii)] \textit{An element $a \in \calA_\sa^+$ is an interior point of $\calA_\sa^+$ if and only if $0$ is not contained in the spectrum of $a$.  
In particular, $\calA_\sa^+$ has interior points if and only if $\calA$ is unital.}
\begin{proof}
Let $a \in \calA_\sa^+$. If $\sigma(a)$ does not contain the number $0$, then 
$\calA$ is unital and we can thus use the Neumann series to prove that $a$ is an interior point of $\calA_\sa^+$. 

Assume conversely that $a$ is an interior point of $\calA_\sa^+$. We first 
prove that $\calA$ is unital. As is well-known, the norm is additive on 
the positive cone of the dual space $(\calA_\sa)'$, so it follows that 
there exists a functional $\eins$ in the bidual of $\calA_\sa$ such that
 $\langle \eins, \tau \rangle = \norm{\tau}$ for each $0 \le \tau \in 
(\calA_\sa)'$. Since the unit ball of a Banach space is always 
weakly${}^*$-dense in the unit ball of the bidual space, it follows that
 $a$ is also an interior point of the cone in $(\calA_\sa)''$; hence, $a
 \ge \delta \eins$ for some $\delta > 0$ by 
Proposition~\ref{prop:characterisation-of-order-units}. 

Consequently, $\langle \tau, a \rangle \ge \delta$ for each state $\tau$ on $\calA$, which shows that $0$ is not in the weak${}^*$ closure of all states on 
$\calA$. As shown in \cite[p.\ 60]{Bratteli1979}, this implies that 
$\calA$ is unital. According to 
Proposition~\ref{prop:characterisation-of-order-units}, the element $a$ 
dominates a strictly positive multiple of $\eins_\calA$, so $0$ is not 
in the spectrum of $a$. 
\end{proof} 
\end{enumerate} 
\end{example} 

\begin{example}[The $C^*$-algebra of compact operators] \label{ex:c-star-algebra-of-compact-operators}  
	Now consider the following special case of Example~\ref{ex:c-star-algebras}: let $H$ be a complex Hilbert space and denote by 
	$\calK(H)$ the $C^*$-algebra of all compact operators on $H$. 
	Then $\calK(H)$ is not unital, so the positive cone in $\calK(H)_\sa^+$ has empty interior. 
	Moreover, an operator $A \in \calK(H)_\sa^+$ is an almost interior point of the positive cone if and only if $A$ has dense 
	range and also if and only if $A$ is injective; for the first equivalence we refer to \cite[Exercise~5(b) on p.\,108]{Murphy1990} 
	and the second equivalence follows, for instance, 
	from the spectral theorem for compact self-adjoint operators.
\end{example}

\begin{example}[Spaces of continuously differentiable functions] 
	Let $k \in \bbN_0$ and let $\Omega \subseteq \bbR^d$ be a non-empty, bounded and open subset and let $C^k(\overline{\Omega})$ 
	denote the space of all real-valued functions on the closure $\overline{\Omega}$ of $\Omega$ which are $k$-times differentiable 
	on $\Omega$ and whose partial derivatives up to order $k$ can be continuously extended to the closure $\overline{\Omega}$. 
	As usual, we endow the space $C^k(\overline{\Omega})$ with the norm given 
	by $\norm{f} := \sum_{\alpha} \norm{f^{(\alpha)}}_{\infty}$, where the sum runs over all multi-indices $\alpha \in \bbN^d$ of length at most $k$ and 
	where $f^{(\alpha)}$ denotes the $\alpha$-th derivative of $f$.
	
	Moreover, we endow $C^k(\overline{\Omega})$ with the cone $C^k_+(\overline{\Omega})$ consisting of all functions $f$ which satisfy $f(\omega) \ge 0$ 
	for all $\omega \in \overline{\Omega}$. Then $C^k(\overline{\Omega})$ is an ordered Banach space with a generating cone, but the cone is not normal 
	unless $k = 0$. Moreover, the cone has non-empty interior and this interior consists of all functions $f$ which satisfy $f(\omega) > 0$ for all 
	$\omega \in \overline{\Omega}$. According to Corollary \ref{cor:almost-interior-points-and-interior-points} the almost interior points 
	coincide with the interior points of $C^k_+(\overline{\Omega})$.
	
	Note that the space $C^k(\overline{\Omega})$ is not a vector lattice unless $k = 0$.
\end{example}

\begin{example}[Sobolev spaces] \label{ex:sobolev-spaces}   
	Let $k \in \bbN_0$, let $p \in [1,\infty)$ and let $W^{k,p}(\bbR^d)$ be the Sobolev space on $\bbR^d$ consisting of all real-valued functions 
	whose distributional derivatives of order at most $k$ are all contained in $L^p(\bbR^d)$ (endowed with the usual Sobolev norm). 
Moreover, endow this space with the cone $W^{k,p}_+(\bbR^d)$ consisting of all functions $f$ which satisfy $f(\omega) \ge 0$ for almost 
all $\omega \in \bbR^d$. 
Then $W^{k,p}(\bbR^d)$ is an ordered Banach space.
	
	The space $W^{k,p}(\bbR^d)$ has a normal cone if and only if $k = 0$; it is lattice ordered if and only if $k \in \{0,1\}$ 
	(the implication ``$\Rightarrow$'' can be shown as in \cite[Example~2.3(d)]{Arendt2009} for Sobolev spaces on bounded subsets 
	of $\bbR^d$). The authors of the recent paper \cite{Ponce2018} showed that the positive cone in $W^{k,p}(\bbR^d)$ is generating by establishing that the cone is even non-flatted.

	We do not know whether the space $W^{k,p}(\bbR^d)$ has almost interior or even quasi-interior points.
\end{example}
\begin{example} \label{ex:extension-of-linear-functionals}
        Almost interior points are related to the possibility of extending a positive functional defined on a subspace of an ordered 
        Banach space to the entire space. 
        We illustrate this by the following example which is taken from \cite{Bakhtin1968}. 
        
	Endow $\ell^2 := \ell^2(\bbN)$ with the usual coordinate-wise order and consider the positive elements 
	\begin{align*}
		x_0 & = (1, \tfrac{1}{2^2},\tfrac{1}{3},\tfrac{1}{4^2},\dots,\tfrac{1}{2n-1},\tfrac{1}{(2n)^2},\dots) \\ 
		\text{and} \qquad y_0 & = (1, \tfrac{1}{2}, \tfrac{1}{3^2},\tfrac{1}{4},\dots, \tfrac{1}{(2n-1)^2},\tfrac{1}{2n},\dots)
	\end{align*}
	in $\ell^2$. It is easy to see that $y_0-\lambda x_0\notin X_+$ for all $\lambda>0$. Let $Y$ denote the $2$-dimensional vector 
	subspace of $\ell^2$ spanned by $x_0$ and $y_0$ and set $Y_+=X_+\cap Y$. Then $(Y,Y_+)$ is an ordered Banach space and the 
	functional $y' \in Y'$ given by
	\begin{align*}
		\langle y', \alpha x_0 + \beta y_0 \rangle := \alpha
	\end{align*}
	for all $\alpha,\beta \in \bbR$ is positive since $y_0-\lambda x_0\notin X_+$ for all $\lambda>0$.
	
	The functional $y'$ cannot be extended to a positive linear functional $x' \in (\ell^2)'$ since such a functional $x'$ would have to 
	map the quasi-interior (and thus almost interior) point $y_0$ of $\ell^2_+$ to $0$.
\end{example}

The simple observation underlying the above example can be upgraded to a theorem which characterises whether a linear functional can be 
extended from a subspace to a positive functional on the entire space. 
As Example \ref{ex:extension-of-linear-functionals}, this result also goes back to Bakhtin \cite[Theorem~1.6]{Bakhtin1968} and \cite{BakhtinGont1968}; 
here we include a version of this result which is a bit more general, although it relies on a similar proof.

\begin{theorem} \label{thm:extending-functionals-via-almost-interior-points}
	Let $(X,X_+)$ be an ordered Banach space and let $Y \subseteq X$ be a vector subspace which contains at least one almost interior point $y_0$ of $X_+$. For every linear map $y': Y \to \bbR$ the following assertions are equivalent:
	\begin{enumerate}[\upshape (i)]
		\item There exists a positive functional $x' \in X' \setminus \{0\}$ which extends $y'$.
		\item The set $X_+ + \ker y'$ is not dense in $X$ and we have $\langle y',y_0\rangle > 0$.
		\item The set $X_+ + \ker y'$ is not dense in $X$ and we have $\langle y',y\rangle > 0$ for every vector $y \in Y$ which is an almost interior point of $X_+$.
	\end{enumerate}
	\end{theorem}

Before the proof a few remarks are in order. We stress that the vector subspace $Y$ is not assumed to be closed in $X$. 
Also note that we did not assume $y'$ to be bounded with respect to any norm, while the functional $x'$ in assertion~(i) is bounded. 
Moreover, $y'$ is not a priori assumed to be positive in the sense that it maps $X_+ \cap Y$ into $[0,\infty)$; nevertheless, 
assertion~(i) implies that $y$ has this property automatically if any of the equivalent assertions~(i)--(iii) holds.  

\begin{proof}[Proof of Theorem~\ref{thm:extending-functionals-via-almost-interior-points}]
	The implication ``(i) $\Rightarrow$ (iii)'' is straightforward to prove. Indeed, assume to the contrary that $X_+ + \ker y'$ is dense 
	in $X$. 
	Then for the positive extension $x'$ of $y'$ we have $\langle x',x\rangle\geq 0$ for all $x\in \overline{X_++\ker y'}=X$.  
	Therefore, $x'=0$, which is impossible. 
	For an arbitrary almost interior point $y\in Y$ one has $0<\langle x',y\rangle =\langle y',y\rangle$.
	
	The implication ``(iii) $\Rightarrow$ (ii)'' is obvious. 
	To prove ``(ii) $\Rightarrow$ (i)'' assume that (ii) is true. We first note that $Y = \ker y' \oplus \bbR y_0$ since $y'$ 
	does not vanish at $y_0$.
	
	Now, let $K$ denote the closure of $X_+ + \ker y'$. Then $K$ is a closed convex subset of $X$ different from $X$. 
	Hence, by the Hahn--Banach theorem there exists a non-zero functional $\tilde x' \in X'$ and a real number $\alpha$ such that 
	$\langle \tilde x',x\rangle \ge \alpha$ for all $x \in K$. 
	Since $K$ is invariant under multiplication by positive scalars, we obtain $\langle \tilde x', x \rangle \ge \frac{\alpha}{n}$ for 
	all $x \in K$ and all $n \in \bbN$, so actually $\langle \tilde x',x\rangle \ge 0$ for all $x \in K$.
	
	As $K$ contains $X_+$ we conclude that $\tilde x'$ is positive. Moreover, $\tilde x'$ vanishes on $\ker y'$, 
	for if $y \in \ker y'$ then $\pm y \in K$ and thus $\langle \tilde x',y\rangle = 0$. 
	
	We conclude the proof by setting $x':= \frac{\langle y',y_0\rangle}{\langle \tilde x',y_0\rangle} \tilde x'$. 
	Observe that the denominator in this term is not zero since $y_0$ is an almost interior point of $X_+$. 
	The functional $x' \in X'$ is positive, it vanishes on $\ker y'$ and it coincides with $y'$ at the point $y_0$. 
	Since $Y = \ker y' \oplus \bbR y_0$ this implies that the restriction of $x'$ to $Y$ coincides with $y'$.
\end{proof}

\subsection{Positive operators and almost interior points}
In this subsection we briefly discuss how almost interior points behave under the action of certain positive operators. 
Those kind of results will be very useful in the proof of our main result in 
Section~\ref{section:strong-convergence-of-positive-semigroups}. 

\begin{proposition} \label{prop:almost-interior-points-by-functional-condition}
	Let $X,Y$ be ordered Banach spaces such that $X_+$ contains an almost interior point. For every positive operator $T \in \calL(X;Y)$ the following assertions are equivalent:
	\begin{enumerate}[\upshape (i)]
		\item There is $x \in X_+$ such that $Tx$ is an almost interior point of $Y_+$. 
        \item $\ker T' \cap Y'_+ = \{0\}$.
        \item $T$ maps almost interior points of $X_+$ to almost interior points of $Y_+$.
		\end{enumerate}
\end{proposition}
 \begin{proof}
 ``(iii) $\Rightarrow$ (i)'' This implication is clear. \\
 ``(i) $\Rightarrow$ (ii)'' For each $y' \in \ker T' \cap Y'_+$ we have 
       $\langle y',Tx\rangle = \langle T'y',x\rangle = 0$; since $Tx$ is an almost 
       interior point of $Y_+$, it follows that $y' = 0$. \\
 ``(ii) $\Rightarrow$ (iii)'' Let $x\in X_+$ be an almost interior point of $X_+$ and 
        let $y' \in Y'$ be a non-zero positive functional. 
According to~(ii) the functional $T'y'$ is 
not zero, so $\langle y', Tx\rangle = \langle T'y', 
x\rangle > 0$. This shows that $Tx$ is an 
almost interior point of $Y_+$. 
\end{proof}

Let $X$ be an ordered Banach space and let $P \in \calL(X)$ be a positive 
projection. Denote the range space of $P$ by $PX$, i.e. $PX:=\{Px\colon x\in X\}$. Then the cone $(PX)_+ := X_+ \cap PX$ 
coincides with the set $P(X_+)$ and the space $(PX, (PX)_+)$ is an ordered Banach space 
in its own right (whose order coincides with the order inherited from $(X,X_+)$).

\begin{corollary} \label{cor:almost-interior-points-by-dense-range}
	Let $X, Y$ be ordered Banach spaces and let $x\in X$ be an almost interior point of $X_+$. 

\begin{enumerate}[\upshape (a)]
\item If $T \in \calL(X,Y)$ is a positive operator with dense range, then $Tx$ is an 
almost interior point of $Y_+$.
\item If $P \in \calL(X)$ is a positive projection, then $Px$ is an almost interior 
point of the positive cone $(PX)_+$ of the ordered Banach space $PX$.
\end{enumerate}
\end{corollary}
\begin{proof}
 (a) Since $T$ has dense range, its dual operator $T'$ is injective, so the assertion follows from implication 
``(ii) $\Rightarrow$ (iii)'' in 
Proposition~\ref{prop:almost-interior-points-by-functional-condition}. 
\\ (b) This is a special case of~(a). 
\end{proof}

\begin{corollary} \label{cor:almost-interior-points-for-semigroups}
	Let $X$ be an ordered Banach space with a positive cone $X_+ \not= \{0\}$, let $J = \bbN_0$ or $J = [0,\infty)$ and 
	let $\calT = (T_t)_{t \in J}$ be a positive operator semigroup on $X$. 
	Let $x$ be a non-zero positive vector in $X$ and let $t_0$ be a non-zero time in $J$ such that $T_{t_0}x$ is an almost interior point of $X_+$. 
	\par
	Then each $T_s$, $s \in J$ maps almost interior points to almost interior points.  
	In particular, $T_tx$ is an almost interior point for every $t \ge t_0$
\end{corollary}

 \begin{proof}
For $0 \le s \le t_0$ the operator $T_s$ 
maps $T_{t_0-s}x$ to the almost interior point $T_{t_0}x$; hence, $T_s$ 
maps almost interior points to almost interior points according to 
Proposition~\ref{prop:almost-interior-points-by-functional-condition}. 
The semigroup law implies that the same remains true for $s \ge t_0$, 
and it thus follows, again by the semigroup law, that $T_tx$ is an 
almost interior point for all $t \ge t_0$. 
 \end{proof}

Almost interior points play also an essential role in the spectral theory of positive operators. 
During the last decades this theory has undergone a considerable development. 
It would probably provide sufficient material and research problems for an article on its own, so we do not discuss it here in detail. 
Instead, for getting a first impression we refer the reader to the classical puplications \cite{Krasnoselskii1989, Stecenko1966, ZabrKrasStec1967}.

\section{Operator Semigroups and the Jacobs--de Leeuw--Glicksberg decomposition} \label{section:jdlg}

In this section we recall a version of the famous Jacobs--de Leeuw--Glicksberg decomposition of operator semigroups. 
This result can be found in various different versions in the literature and is very helpful if one wants to prove convergence results 
for semigroups under compactness assumptions on its orbits. 
In Theorem~\ref{thm:jdlg}  we show a version of the theorem which is particularly well-adapted 
for our application in Section~\ref{section:strong-convergence-of-positive-semigroups}. 
 We use the following concept which is taken from \cite[p.\,2636]{Emelyanov2001}:

 \begin{definition}
Let $J = \bbN_0$ or $J = [0,\infty)$. 
An operator semigroup $\calT = (T_t)_{t \in J}$ is called \emph{strongly
 asymptotically compact} if for each $x
 \in X$ and each sequence $(t_n)_{n \in \bbN} \subseteq J$ that 
converges to $\infty$, the sequence $(T_{t_n}x)_{n \in \bbN}$ in $X$ has
 a convergent subsequence.
\end{definition}
If for a semigroup $\calT = (T_t)_{t \in J}$ the orbit $\{T_t x: 
\, t \in J\}$ is relatively compact in $X$ for each $x\in X$, then the semigroup is 
strongly asymptotically compact. Moreover, if $J = \bbN_0$ or if $J = 
[0,\infty)$ and $\calT$ is a $C_0$-semigroup, then it is easy to see 
that relative compactness of all orbits is equivalent to strong 
asymptotic compactness of the semigroup. However, if $J = [0,\infty)$ 
and the semigroup is not strongly continuous, then strong asymptotic 
compactness is a more general notion than relative compactness of the 
orbits (since strong asymptotic compactness allows the semigroup to 
behave quite wildly for small times). For characterizations of strong 
asymptotic compactness we refer to \cite[Proposition~2.5]{Glueck2019}. 
In the same paper, it is also explained why this property is closely 
related to the long term behaviour of operator semigroups. In the 
present article, we need the following  result which is our adapted version of the Jacobs--de Leeuw--Glicksberg Theorem and turns out to be a consequence of 
\cite[Theorem~2.2]{Glueck2019}.
\begin{theorem} \label{thm:jdlg}
Let $J = \bbN_0$ or $J = [0,\infty)$ 
and let $\calT = (T_t)_{t \in J}$ be a bounded and positive operator 
semigroup on an ordered Banach space $X$. Assume that $\calT$ is 
strongly asymptotically compact.

Then there exist a positive projection
 $P \in \calL(X)$ and a subgroup $\calG$ of the invertible operators in 
$\calL(P X)$ with the following properties:
\begin{enumerate}[\upshape (a)] 
\item The projection $P$ commutes with the operator $T_t$ for each $t \in J$, 
      i.e.\ both the range and the kernel of $P$ are left invariant by the semigroup $\calT$.
\item For every $x \in \ker P$ the vector $T_tx$ converges to $0$ with respect to 
      the norm on $X$ as $t \to \infty$.
\item The group $\calG$ is strongly compact and contains only positive operators.
\item The restriction of each operator $T_t$, $t \in J$, to $P X$ is 
      contained in $\calG$.
\end{enumerate}
\end{theorem}

\begin{proof} 
Define $\calT_\infty := \bigcap_{s \in
 J} \overline{\{T_t: \, s \le t \in J\}}$, where the closure is taken in
 the strong operator topology. According to \cite[Proposition~2.5(i) 
and~(vi)]{Glueck2019} the assumptions on our semigroup imply that 
$\calT_\infty$ is non-empty and strongly compact. Hence we can apply 
\cite[Theorem~2.2]{Glueck2019} which yields a projection $P \in 
\calT_\infty$ (denoted by $P_\infty$ in \cite[Theorem~2.2]{Glueck2019}) 
and a group $\calG$, given as the strong closure of $\{T_t|_{PX},: \, t 
\in J\}$ in $\calL(PX)$, with the desired properties. 
Note that the positivity of $P$ follows from $P \in \calT_\infty$, and the positivity 
of the elements of $\calG$ follows from the positivity of the operators 
$T_t|_{P X}$.
\end{proof}

In the above theorem we employed the concept of the 
\emph{semigroup at infinity} $\calT_\infty$ which was recently studied 
in \cite{Glueck2019}. More classical approaches apply the Jacobs--de 
Leeuw--Glicksberg theory directly to the strong closure of $\{T_t:\, t 
\in J\}$, which requires the stronger assumption that all orbits of the 
semigroup be relatively compact. For further information about the 
Jacob-de Leeuw-Glicksberg decomposition of operator semigroups we refer 
to \cite[Section~2.4]{Krengel1985}, \cite[Chapter~16]{Eisner2015} and, 
within the context of $C_0$-semigroups, to 
\cite[Section~V.2]{Engel2000}.

\section{Strong Convergence of Positive Semigroups} \label{section:strong-convergence-of-positive-semigroups}

Let us start right away with the main result of this section:
\begin{theorem} \label{thm:strong-convergence}
	Let $X$ be an ordered Banach space with positive cone $X_+ \not= \{0\}$, let $J = \bbN_0$ or $J = [0,\infty)$ and 
	let $\calT = (T_t)_{t \in J}$ be a positive operator semigroup on $X$. 
	Suppose that the following two assumptions are satisfied:
	\begin{enumerate}[\upshape (a)]
		\item For each vector $0 \not= x \in X_+$ there exists a time $t_x \in J$ such that $T_{t_x} x$ is an 
		almost interior point of $X_+$.
		\item The semigroup $\calT$ is bounded and strongly asymptotically compact. 
	\end{enumerate}
	Then $T_t$ converges strongly as $t \to \infty$. 
	If the limit operator $Q := \lim_{t \to \infty} T_t$ is not zero, 
	then it is of the 
	form $Q = y' \otimes y$. Here, $y$ is an almost interior point of $X_+$ and a fixed point of $\calT$, and $y' \in X'$ is a 
	strictly positive functional and a fixed point of the adjoint semigroup 
	$\calT' := (T_t')_{t\in J}$ (consisting of the dual operators $T_t'$).
\end{theorem}

Before we prove Theorem~\ref{thm:strong-convergence} let us point out that the theorem imposes no special conditions on the 
positive cone $X_+$ except for the assumption that it be non-zero. In particular, $X_+$ is not assumed to be normal nor is 
it assumed to be generating. 
However, we should note that the assumptions of the theorem clearly imply the existence of an almost interior point of $X_+$, 
so the assumptions can only be satisfied if the positive cone $X_+$ is at least total 
(see Proposition~\ref{prop:almost-interior-points-imply-total-cone}).

We mention here that under much stronger conditions (the cone $X_+$ is supposed to be normal and to have non-empty interior) 
the theorem was proved by Makarow and Weber in \cite[Theorems~1 and~4]{Makarow2000}, see also \cite[p. 207]{Wulich2017}.
\begin{proof}[Proof of Theorem~\ref{thm:strong-convergence}]
Due to assumption~(b) we can apply 
Theorem~\ref{thm:jdlg}; let $P$ and $\calG$ be as in this theorem. 
If $P = 0$ there is nothing to show, so we may assume that $PX \not= \{0\}$. 
It follows from assumption~(a) and Proposition~\ref{prop:almost-interior-points-imply-total-cone} that $X_+$ is total in $X$. 
Hence, $(PX)_+$ is total in $PX$ and in 
particular, $(PX)_+ \not= 0$. 
\par We now show that 
\begin{enumerate}[\upshape (1)]
 \item  $PX_+ \setminus \{0\}$ consists of almost-interior points of $X_+$, 
\item  $PX$ is $1$-dimensional, 
\item  $\calT$ acts as the identity on $PX$, 
\item  $T_t \to P$ strongly and 
\item  $P$ is of the claimed form. 
\end{enumerate}
To this end, let $0 \not= x \in (PX)_+$. By 
Corollary~\ref{cor:almost-interior-points-by-dense-range}(b), $T_{t_x}x$
 is an almost interior point not only of $X_+$ but also of $(PX)_+$. 
 We note that $G T_{t_x}|_{PX} = P$ for some $G \in \calG$ since $T_{t_x}$ 
is an element of $\calG$, and the latter set is a subgroup of the 
invertible operators on $PX$. In particular, $G$ maps $T_{2t_x}x$ to the
 almost interior point $T_{t_x}x$ of $X_+$, so 
Proposition~\ref{prop:almost-interior-points-by-functional-condition} 
shows that $G$ maps almost interior points of $(PX)_+$ to almost 
interior points of $X_+$. Consequently, $x = GT_{t_x}x$ is an almost 
interior point of $X_+$, which shows~(1). 
\par
By employing~(1) and again 
Corollary~\ref{cor:almost-interior-points-by-dense-range}(b) we see that
 every point in $PX_+ \setminus \{0\}$ is also an almost interior point 
of $(PX)_+$, so~(2) follows from 
Theorem~\ref{thm:existence-of-non-almost-interior-points}. 
To show~(3), let $0 < y \in PX$ be an element that spans $PX$. 
Then $P$ is of the form $P = y' \otimes y$ for some $y' \in X' \setminus \{0\}$, where 
$\langle y',y\rangle = 1$. As $y$ and $P$ are positive, so is $y'$. 
For each $G \in \calG$ we have $Gy = \langle y',Gy\rangle y$, so the mapping 
\begin{align*}
 \calG \ni G \mapsto \langle y', Gy\rangle \in (0,\infty)
\end{align*} 
 is a (strongly continuous) group homomorphism; as $\calG$ is strongly compact, 
 its image is a compact subgroup of the multiplicative group $(0,\infty)$, i.e.\ the image is 
equal to $\{1\}$. Hence, $Gy = y$ for all $G \in \calG$, which 
shows~(3). 
\par 
As $T_t$ converges strongly to $0$ on $\ker P$ (see Theorem~\ref{thm:jdlg}(b)), 
(4) follows from~(3). 
\par It remains to show~(5), and we have already seen in~(3) that $y$ is a fixed point of 
$\calT$ and in~(1) that $y$ is a quasi-interior point of $X_+$. 
To see that $y'$ is a fixed point of the dual operator semigroup $\calT'$ observe that, 
for each $t \in J$ and each $x \in X$, 
\begin{align*} 
 \langle T_t'y',x\rangle y = \langle y', T_tx\rangle y = PT_tx = T_tPx = Px
 = \langle y', x\rangle y,
\end{align*}
where we have used $P=y' \otimes y$ for the second and the last equation. 
So $\langle T_t'y',x\rangle = \langle y',x\rangle$ and hence, $T_t'y' = y'$. 
Finally, to see that $y'$ is strictly positive, note that for $x \in X_+ \setminus \{0\}$
\begin{align*}
 \langle y', x\rangle = \langle T_{t_x}'y',x\rangle = \langle y', T_{t_x}x\rangle > 0
 \end{align*} 
since $T_{t_x}x$ is an almost interior point of $X_+$. 
We have thus shown the theorem with $Q:= P$. 
\end{proof}
The following corollary illustrates that the assumption (a) of the theorem is sometimes satisfied as a consequence of a perturbation.   
The corollary assumes knowledge of the basic theory of $C_0$-semigroups 
(see e.g.\ \cite{Engel2000} for a comprehensive 
treatment of this theory). 
We use the notation $(e^{tA})_{t \in [0,\infty)}$ for a $C_0$-semigroup with generator $A$.
\begin{corollary} \label{cor:perturbed-semigroup} 
Let $X$ be an ordered Banach space and let $(e^{tA})_{t \in [0,\infty)}$ be a positive and contractive\footnote{\, This  means:  $\norm{e^{tA}} 
\le 1$ for all $t \in [0,\infty)$.} $C_0$-semigroup on $X$ whose generator $A: X \supseteq D(A) \to X$ has compact resolvent. 
Let $B \in \calL(X)$ be a positive operator such that $Bx$ is an almost interior point of $X_+$ for each $0\neq x \in X_+$ and let $c \ge \norm{B}$. 
Then $A+B-cI$ generates a positive $C_0$-semigroup $\big(e^{t(A+B-cI)}\big)_{t \in [0,\infty)}$ 
which converges strongly as $t \to \infty$. 
\end{corollary}

Note that the assertion of the corollary is trivial if $c > \norm{B}$, i.e.\ the interesting 
case is $c = \norm{B}$.
\begin{proof}[Proof of Corollary~\ref{cor:perturbed-semigroup}]
First note that $A+B-cI$ and $A+B$ generate $C_0$-semigroups by standard perturbation 
theory (see e.g.~\cite[Theorem~III.1.3]{Engel2000}). 
We have \begin{align*} 
         e^{t(A+B-cI)} = e^{-ct}e^{t(A+B)}\;\text{ for all } t \ge 0.
        \end{align*}
 It follows from the 
Dyson--Phillips series representation of perturbed $C_0$-semigroups (see
 e.g.\ \cite[Theorem~III.1.10]{Engel2000}) that $e^{t(A+B)}$ is positive
 for each $t \ge 0$ and hence, so is $e^{t(A+B-cI)}$.
\par
 We now check that both assumptions~(a) 
and~(b) of Theorem~\ref{thm:strong-convergence} are satisfied. 
To see~(a), we again use the Dyson--Phillips series for $e^{t(A+B)}$ which 
shows that 
\begin{align*}
e^{t(A+B)} \ge \int_0^t e^{(t-s)A}Be^{sA} \dx s
\end{align*} 
for all $t \in [0,\infty)$, where the integral is to be understood in the strong sense. 
Now, let $0\neq x \in X_+$. By strong continuity we can choose a time $t_x > 0$
 such that $e^{t_xA}x \not= 0$. For each $0\neq x' \in X'_+$, 
the continuous mapping 
 \begin{align*}
 [0,t_x] \ni s \mapsto \langle x', e^{(t_x-s)A}Be^{sA}x\rangle \in [0,\infty) 
 \end{align*} 
is not zero at $s = t_x$ since $Be^{t_xA}x$ is an almost interior point of $X_+$. 
Hence, 
\begin{align*}
 \langle x',e^{t_x(A+B)}x\rangle \ge \int_0^{t_x} \langle x', 
e^{(t_x-s)A}Be^{sA}x\rangle \dx s > 0,
\end{align*}
 which shows that $e^{t_x(A+B)}x$ is an almost interior point of $X_+$. 
 Consequently, so is $e^{t_x(A+B-cI)}x$. 
\par 
To show that assumption~(b) of Theorem~\ref{thm:strong-convergence} is satisfied, 
first note that $A+B-cI$ has compact resolvent as its domain coincides with the domain 
of $A$. Since $A$ is dissipative\footnote{\, This means: 
$\norm{(\lambda\id-A)x} \geq \lambda\norm x$ for all $\lambda>0$ and 
$x\in D(A)$; see \cite[Proposition~II.3.23]{Engel2000} for a useful 
characterisation of dissipativity.}
 and since $c \ge \norm{B}$, it follows that 
$A+B-cI$ is also dissipative, so the semigroup $\big(e^{t(A+B-cI)}\big)_{t \in [0,\infty)}$ 
is contractive (this is a simple special case of the 
Lumer--Phillips theorem, see for instance \cite[Theorem~II.3.15]{Engel2000}). 
But boundedness of a $C_0$-semigroup together with compactness of the resolvent implies 
that all orbits of the semigroup are relatively compact (see for instance 
\cite[Corollary~V.2.15(i)]{Engel2000}). 
\par Thus, Theorem~\ref{thm:strong-convergence} is applicable and implies the assertion.
\end{proof}

\section{Uniform Convergence of Positive Semigroups} \label{section:uniform-convergence-of-positive-semigroups}

The main theorem of this section is Theorem~\ref{thm:uniform-convergence} below. It is a version of Theorem~\ref{thm:strong-convergence}, 
but yields  -- due to stronger assumptions -- convergence with respect to the operator norm. 

Let $T \in \calL(X)$ for a Banach space $X$. The operator $T$ is called \emph{power bounded} if $\sup_{n \in \bbN_0} \norm{T^n} < \infty$, and $T$ is called \emph{quasi-compact} if there exists an integer $m \in \bbN$ and a compact linear operator $K$ on $X$ such that $\norm{T^m - K} < 1$.
In the following proposition we list a number of properties of quasi-compact operators which we are going to use in the sequel; these properties are certainly not new, but we collect them here in concise way for later reference. 
For a real or complex Banach space $X$ we let  $\calK(X) \subseteq \calL(X)$ denote the closed ideal of all compact linear operators on $X$. 
The quotient Banach algebra $\calL(X) / \calK(X)$, endowed with the usual quotient norm 
$\norm{[T]}:= \inf\{\norm{T-K}\colon K\in \calK(X)\}$ 
for the equivalence class $[T]=T+\calK(X)$ of the operator $T$, is called the \emph{Calkin algebra} over $X$. 
The latter is very useful in the following proposition.

\begin{proposition} \label{prop:quasi-compact-operators}
	Let $X$ be a Banach space and let $T \in \calL(X)$.
	\begin{enumerate}[\upshape (a)]
		\item Let $[T]$ denote the equivalence class of $T$ in the  Calkin algebra  \linebreak 
		$\calL(X)/\calK(X)$. The operator $T$ is quasi-compact if and only if $\,[T]^n \to 0$ with respect to the quotient norm on $\calL(X)/\calK(X)$ as $n \to \infty$.
		\item Let $T$ be quasi-compact and suppose that $S \in \calL(X)$ is power bounded and commutes with $T$. Then $TS$ is quasi-compact, too.
		\item Let $k \in \bbN$. Then $T$ is quasi-compact if and only if $T^k$ is quasi-compact. 
		\item Let $J = \bbN_0$ or $J = [0,\infty)$ and let $(T_t)_{t \in J}$ be a bounded operator semigroup on $X$. If $T_{t_0}$ is 
		quasi-compact for one time $t_0 \in J$, then $T_t$ is quasi-compact for each time $t \in J \setminus \{0\}$.
	\end{enumerate}
\end{proposition}
\begin{proof}
	(a) By definition, $T$ is quasi-compact if and only if there exists $m \in \bbN_0$ such that $\norm{[T]^m} =\norm{[T^m]} < 1$, 
            which is in turn equivalent to $\norm{[T]^n} \to 0$ as $n \to \infty$.
	
	(b) As $S$ is power bounded, we have $M := \sup_{n \in \bbN_0} \norm{S^n} < \infty$. Moreover,
	\begin{align*}
		\norm{[TS]^n} = \norm{[T^n][S^n]} \le \norm{[T]^n} \norm{[S]^n} \le \norm{[T]^n} M \to 0 \quad \text{for } n \to \infty;
	\end{align*}
	for the first equality we used that $S$ commutes with $T$. It follows from~(a) that $TS$ is quasi-compact.
	
	(c) The implication ``$\Leftarrow$'' follows right from the definition of quasi-com\-pactness, and the converse 
	implication ``$\Rightarrow$'' follows from~(a) since $[T^k]^n = [T]^{kn}$ for all $k,n \in \bbN$.
	
	(d) Let $t \in J \setminus \{0\}$ and choose $n \in \bbN$ such that $nt \ge t_0$. 
	The operator $T_t^n = T_{nt} = T_{t_0}T_{nt - t_0}$ is quasi-compact according to~(b), so it follows from~(c) that $T_t$ 
	is quasi-compact, too. 
\end{proof}
A second ingredient for the proof of Theorem~\ref{thm:uniform-convergence} is the next proposition about the orbits of quasi-compact operators. Again, this is far from being 
new (for instance, the proposition can easily be derived from the spectral 
representation of quasi-compact operators in \cite[Theorem 2.8 on page 
91]{Krengel1985}), but for the convenience of the reader, and also to be
 more self-contained, we include an elementary proof.

\begin{proposition} \label{prop:consequences-of-quasi-compactness}
	Let $X$ be a Banach space, let $T \in \calL(X)$ and assume that $T$ is quasi-compact.
	\begin{enumerate}[\upshape (a)]
		\item If $T$ is power-bounded, then the orbit $\{T^nx: \, n \in\bbN_0\}$ is relatively compact 
                      in $X$ for each $x \in X$.
		\item If $T^n$ converges strongly to an operator $P \in \calL(X)$ as $n \to \infty$, 
	then $T^n$ even converges with respect to the operator norm to $P$ as $n \to \infty$.
	\end{enumerate}
\end{proposition}

For the proof of assertion~(b) we need the following elementary observation: 
if a sequence of operators $(T_n)_{n \in \bbN_0} \subseteq \calL(X)$ on a Banach space $X$ converges strongly to an operator 
$S \in \calL(X)$ and if $C \subseteq X$ is a relatively compact set, then the convergence of $T_n$ to $S$ is uniform on $C$, 
meaning that $\sup_{x \in C} \norm{T_nx - Sx} \to 0$ as $n \to \infty$.

\begin{proof}[Proof of Proposition~\ref{prop:consequences-of-quasi-compactness}]
	Each of the assumptions (a) and (b) implies that $T$ is power-bounded. 
	Set $M := \sup_{n \in \bbN_0} \norm{T^n} < \infty$.
	
	(a) Fix a vector $x \in X$, say of norm $1$. In view of the norm-completeness of $X$ it suffices to show that the orbit $\{T^nx: \, n \in\bbN_0\}$ is totally bounded, 
	so let $\varepsilon > 0$. As $T$ is quasi-compact, it follows from Proposition~\ref{prop:quasi-compact-operators}(a) that 
	there exists an integer $m \in \bbN$ and a compact operator $K \in \calL(X)$ such that $\norm{T^m - K} \le \varepsilon$. 
	For each integer $n \ge m$ we have
	\begin{align*}
		T^nx = (T^m-K)T^{n-m}x + KT^{n-m}x,
	\end{align*}
	where the first sumand $(T^m-K)T^{n-m}x$ has norm at most $\varepsilon M$, and \linebreak the second sumand $KT^{n-m}x$ 
	is contained in the relatively compact set $M K(\overline{B})$, where $\overline{B}$ denotes the closed unit ball in $X$. 
	Since the set $M K(\overline{B})$ can be covered by finitely many balls of radius $\varepsilon$, it follows that the 
	set $\{T^nx: \, n \ge m\}$, and hence also the entire orbit $\{T^nx: \, n \in \bbN\}$, can be covered by finitely many balls
	of radius $\varepsilon(M+1)$. This proves that the orbit is totally bounded and thus relatively compact.
	
	(b) It suffices to prove that $(T^n)_{n \in \bbN_0}$ is a Cauchy sequence with respect to the operator norm, so let $\varepsilon > 0$. As above, we can find an integer $m \in \bbN$ and a compact operator $K \in \calL(X)$ such that $\norm{T^m - K} \le \varepsilon$. Since the sequence $(T^n)_{n \in \bbN_0}$ converges strongly to $P$, it follows from the compactness of $K$ that $T^nK$ converges to $PK$ with respect to the operator norm as $n \to \infty$. In particular, $(T^nK)_{n \in \bbN_0}$ is a Cauchy sequence, so we can find $n_0 \in \bbN$ such that $\norm{T^{n_1}K - T^{n_2}K} \le \varepsilon$ for all $n_1,n_2 \ge n_0$. For $n_1,n_2 \ge n_0 + m$ we conclude that
	\begin{align*}
		\norm{T^{n_1}- T^{n_2}} \leq  \norm{(T^{n_1-m} - T^{n_2-m})(T^m-K)} + \norm{(T^{n_1-m} - T^{n_2-m}) K} & \\
		\le 2M\varepsilon & + \varepsilon,
	\end{align*}
	which shows that $(T^n)_{n \in \bbN_0}$ is indeed a Cauchy sequence.
\end{proof}

With a few more arguments one can show that in Proposition~\ref{prop:consequences-of-quasi-compactness}(b) it is actually sufficient to 
assume that $T^n$ converges to $P$ with respect to the \emph{weak} operator topology as $n \to \infty$. However, since this is not 
needed for the proof of our main result, we do not discuss this in detail.

Now we can prove the main result of this section.

\begin{theorem} \label{thm:uniform-convergence}
	Let $X$ be an ordered Banach space with positive cone $X_+ \not= \{0\}$, let $J = \bbN_0$ or $J = [0,\infty)$ and 
	let $\calT = (T_t)_{t \in J}$ be a positive operator semigroup on $X$. Suppose that the following two assumptions 
	are satisfied:
	\begin{enumerate}[\upshape (a)]
		\item For each vector $0 \not= x \in X_+$ there exists a time $t_x \in J$ such that $T_{t_x} x$ is an almost interior point 
		of $X_+$.
		\item The semigroup $\calT$ is bounded and there exists a time $\tau \in J$ such that $T_\tau$ is quasi-compact.
	\end{enumerate}

	Then $T_t$ converges with respect to the operator norm as $t \to \infty$. 
	If the limit operator $Q := \lim_{t \to \infty} T_t$ is not zero, then it is of the 
	form $Q = y' \otimes y$. Here, $y$ is an 
	almost interior point of $X_+$ and a fixed point of $\calT$, and $y' \in X'$ is a strictly positive functional and a fixed point 
	of the dual semigroup $\calT' := (T_t')_{t\in J}$.
\end{theorem}

A special case of this result was proved by Makarow and Weber in \cite[Theorems~3 and~5]{Makarow2000}.

\begin{proof}[Proof of Theorem~\ref{thm:uniform-convergence}]
	According to Proposition~\ref{prop:quasi-compact-operators}(d), the operator $T_t$ is quasi-compact for each $t \in J \setminus \{0\}$.

	For the proof of the theorem we distinguish the two cases $J = \bbN_0$ and $J = [0,\infty)$. 
	First let $J = \bbN_0$. 
	Then $T_t = T_1^t$ for all $t \in J = \bbN_0$, and it follows from Proposition~\ref{prop:consequences-of-quasi-compactness}(a)
	and Theorem~\ref{thm:strong-convergence} that  
    $T_t$ converges strongly to an operator $Q$ with the claimed properties as $t \to \infty$. 
	Proposition~\ref{prop:consequences-of-quasi-compactness}(b) then 
	implies that the convergence in fact takes place with respect to the operator norm. 
	  
		Now we consider the case $J = [0,\infty)$ and we reduce it to the time-discrete case, 
		so fix $t \in (0,\infty)$. 
It follows from Corollary~\ref{cor:almost-interior-points-for-semigroups} that the 
time-discrete operator semigroup $(T_{nt})_{n \in \bbN_0} = (T_t^n)_{n 
\in \bbN_0}$ satisfies all assumptions of the current theorem, so we conclude that $T_t^n$ converges with respect to the operator norm to an operator $Q_t \in \calL(X)$, and 
$Q_t$ is of the form claimed for the operator $Q$ in the theorem. 
We can now employ \cite[Theorem~2.1(b)]{GerlachLB} which shows that $Q := Q_t$ is actually
 independent of $t$ and that $T_t$ converges to $Q$ with respect 
to the operator norm as $t \to \infty$.
\end{proof}

\subsection*{Acknowledgements}

The authors are very grateful to the referees for many valuable comments and remarks, which enabled us to improve the presentation of the article, 
to generalize several results and to bring several results and proofs into a more concise form.

The first named author is indebted to the members of the Institute of Analysis at Technische Universit\"at Dresden for their kind invitation to a short stay 
during which some of the work on this article was done.


\begin{thebibliography}{10}

\bibitem{Abdelaziz1975}
Nazar~H. {Abdelaziz}.
\newblock {\em Spectral properties of some positive operators in a Banach space
  with the decomposition property.}
\newblock Proc. Am. Math. Soc. 48 (1975), 344--350.

\bibitem{Alekhno2013}
Egor~A. {Alekhno}.
\newblock {\em Characterization of Banach lattices in terms of quasi-interior
  points.}
\newblock Int. J. Math. Math. Sci. 11 (2013), Art.\ ID 507482.

\bibitem{Aliprantis2007}
Charalambos~D. Aliprantis and Rabee Tourky.
\newblock {\em Cones and duality}, volume~84 of {\em Graduate Studies in
  Mathematics}.
\newblock American Mathematical Society, Providence, RI, 2007.

\bibitem{Arendt1986}
W.~Arendt, A.~Grabosch, G.~Greiner, U.~Groh, H.~P. Lotz, U.~Moustakas,
  R.~Nagel, F.~Neubrander, and U.~Schlotterbeck.
\newblock {\em One-parameter semigroups of positive operators}, volume 1184 of
  {Lecture Notes in Mathematics}.
\newblock Springer-Verlag, Berlin, 1986.

\bibitem{Arendt2009}
Wolfgang Arendt and Robin Nittka.
\newblock {\em Equivalent complete norms and positivity.}
\newblock Arch. Math. (Basel) 92 (2009), 414--427.

\bibitem{Bakhtin1968}
I.~A. Bakhtin.
\newblock {\em Extension of linear positive functionals.}
\newblock Sib. Math. J. 9 (1969) 359--366.

\bibitem{BakhtinGont1968}
I.~A. Bakhtin and G.~M. Gontcharov.
\newblock {\em Criteria for the extension of linear positive functionals.}
\newblock Izv. Vuzov Math. no.11 (1968), 12-18.

\bibitem{Bartoszek1992}
Wojciech Bartoszek.
\newblock {\em Riesz decomposition property implies asymptotic periodicity of
  positive and constrictive operators.}
\newblock Proc. Amer. Math. Soc. 116 (1992) 101--111.

\bibitem{Batty1984}
Charles J.~K. Batty and Derek~W. Robinson.
\newblock {\em Positive one-parameter semigroups on ordered {B}anach spaces.}
\newblock Acta Appl. Math. 2 (1984) 221--296.

\bibitem{Bishop1963}
E.~{Bishop} and R.R. {Phelps}.
\newblock {\em The support functionals of a convex set}.
\newblock In Proc. Sympos. Pure Math. 7 (1963), 27--35.

\bibitem{Bratteli1979}
Ola~Bratteli and Derek~W. Robinson.
\newblock {\em $C^*$- and $W^* $-algebras, symmetry groups, decomposition of states}. 
\newblock {\em Operator algebras and quantum statistical mechanics. {V}ol. 1}. 
\newblock Springer-Verlag, New York-Heidelberg, 1979.

\bibitem{Brezis1976}
Ha\"{\i}m {Br\'ezis} and Felix~E. {Browder}.
\newblock {\em A general principle on ordered sets in nonlinear functional
  analysis.}
\newblock Adv. Math. 21 (1976), 355--364.

\bibitem{Danes1972}
Josef {Danes}.
\newblock {\em A geometric theorem useful in nonlinear functional analysis.}
\newblock Boll. Unione Mat. Ital., IV. Ser. 6 (1972), 369--375.

\bibitem{Ding2003}
Yiming Ding.
\newblock {\em The asymptotic behavior of {F}robenius-{P}erron operator with local
  lower-bound function.}
\newblock Chaos Solitons Fractals 18 (2003), 311--319.

\bibitem{Ei36} 
M.~Eidelheit.
{\em Zur Theorie der konvexen Mengen in linearen normierten Räumen.}
\newblock {Studia Mathematica} vol. VI (1936), 104--111.

\bibitem{Eisner2015}
Tanja Eisner, B\'alint Farkas, Markus Haase, and Rainer Nagel.
\newblock {\em Operator theoretic aspects of ergodic theory}, volume 272 of
  {Graduate Texts in Mathematics}.
\newblock Springer, Cham, 2015.

\bibitem{Emelyanov2001}
E. Yu.~Emel'yanov, U. Kohler, F.~R{\"a}biger and M. P. H.~Wolff.
\newblock {\em Stability and almost periodicity of asymptotically dominated semigroups of positive operators.}
\newblock Proc. Am. Math. Soc. 129 (2001), 2633–2642. 

\bibitem{Engel2000}
Klaus-Jochen Engel and Rainer Nagel.
\newblock {\em One-parameter semigroups for linear evolution equations}, volume
  194 of {Graduate Texts in Mathematics}.
\newblock Springer-Verlag, New York, 2000.
\newblock With contributions by S. Brendle, M. Campiti, T. Hahn, G. Metafune,
  G. Nickel, D. Pallara, C. Perazzoli, A. Rhandi, S. Romanelli and R.
  Schnaubelt.

\bibitem{ErkursunOezcan2018}
Nazife Erkur{\c{s}}un~{\"O}zcan and Farrukh Mukhamedov.
\newblock {\em Uniform ergodicity of Lotz--R{\"a}biger nets of Markov operators on
  abstract state spaces.}
\newblock Results in Mathematics 73 (2018), Art.\ 35.

\bibitem{Gerlach2013}
M.~Gerlach.
\newblock {\em On the peripheral point spectrum and the asymptotic behavior of
  irreducible semigroups of {H}arris operators.}
\newblock Positivity 17 (2013), 875--898.

\bibitem{GerlachConvPOS}
Moritz Gerlach and Jochen Gl\"uck.
\newblock {\em Convergence of positive operator semigroups.}
  \newblock {Trans. Amer. Math. Soc.}, electronically published on June 17, 2019,
  DOI: 10.1090/tran/7836 (to appear in print).

\bibitem{GerlachLB}
Moritz Gerlach and Jochen Gl{\"u}ck.
 \newblock {\em Lower bounds and the asymptotic behaviour of positive operator semigroups}.
  \newblock Ergodic Theory Dynam. Systems 38 (2018), 3012--3041.

\bibitem{Gerlach2017}
Moritz Gerlach and Jochen Gl\"uck.
\newblock {\em On a convergence theorem for semigroups of positive integral
  operators.}
\newblock C. R. Math. Acad. Sci. Paris 355 (2017), 973--976.

\bibitem{GlueckCStar}
Jochen Gl\"{u}ck.
\newblock {\em A note on lattice ordered $C^*$-algebras and Perron--Frobenius
  theory.}
\newblock Mathematische Nachrichten 291 (2018), 1727--1732.

\bibitem{GlueckDISS}
Jochen Gl{\"u}ck.
\newblock {\em Invariant Sets and Long Time Behaviour of Operator Semigroups}.
\newblock PhD thesis, Universit\"at Ulm, 2016.
\newblock DOI: 10.18725/OPARU-4238.

\bibitem{Glueck2019}
Glück J., Haase M.
\newblock {\em Asymptotics of Operator Semigroups via the Semigroup at Infinity.}
\newblock In: Positivity and Noncommutative Analysis,
\newblock Trends in Mathematics. Birkhäuser, Cham (2019)

\bibitem{GlueckWolffLB}
Jochen Gl\"uck and Manfred P.~H. Wolff.
\newblock {\em Long--term analysis of positive operator semigroups via asymptotic domination}.
\newblock Positivity 23 (2019), 1113--1146.

\bibitem{Greiner1982}
G.~Greiner.
\newblock {\em Spektrum und {A}symptotik stark stetiger {H}albgruppen positiver
  {O}peratoren.}
\newblock {\em Sitzungsber. Heidelb. Akad. Wiss. Math.-Natur. Kl.} (1982), 55--80.

\bibitem{Jameson1970}
G. Jameson.
\newblock {\em {Ordered {L}inear {S}paces}}.
Lecture Notes in Math. 141. 
\newblock Springer-Verlag, Berlin-Heidelberg-NewYork, 1970.

\bibitem{KP06} Katsikis, Vasilios; Polyrakis, Ioannis A.
\newblock Positive bases in ordered subspaces with the Riesz decomposition property.
\newblock {\em Studia Mathematica}, 174(3):233--253, 2006.

\bibitem{Keicher2006}
V.~Keicher.
\newblock {\em On the peripheral spectrum of bounded positive semigroups on atomic
  {B}anach lattices.}
\newblock Arch. Math. (Basel) 87 (2006), 359--367.

\bibitem{Krasnoselskii1960}
M.~ A. Krasnosel'skii.
\newblock {\em Regular and completely regular cones.}
\newblock Dokl. Akad. Nauk SSSR 135(2) (1960), 255--257.

\bibitem{Krasnoselskii1989}
M.~A. Krasnosel'skii, Je.~A. Lifshits and A.~V. Sobolev.
\newblock {\em {Positive Linear System - The Method of Positive Operators -}}. 
\newblock  Heldermann Verlag, Berlin, 1989.
\newblock  Transl. from Russian by J\"urgen Appell.

\bibitem{Krein1950}
M.~G. Krein and M.~A. Rutman.
\newblock {\em Linear operators leaving invariant a cone in a {B}anach space.}
\newblock Amer. Math. Soc. Translation 26 (1950), 128 pages.

\bibitem{Krengel1985}
Ulrich Krengel.
\newblock {\em Ergodic theorems}, volume~6 of {De Gruyter Studies in
  Mathematics}.
\newblock Walter de Gruyter \& Co., Berlin, 1985.
\newblock With a supplement by Antoine Brunel.

\bibitem{Kulik2015}
Alexei Kulik and Michael Scheutzow.
\newblock {\em A coupling approach to {D}oob's theorem.}
\newblock Atti Accad. Naz. Lincei Rend. Lincei Mat. Appl. 26 (2015), 83--92.

\bibitem{Lasota1982}
A.~Lasota and James~A. Yorke.
\newblock {\em Exact dynamical systems and the {F}robenius-{P}erron operator.}
\newblock Trans. Amer. Math. Soc. 273 (1982), 375--384.

\bibitem{Lotz1986}
Heinrich~P. Lotz.
\newblock {\em Positive linear operators on {$L^p$} and the {D}oeblin condition.}
\newblock In {Aspects of positivity in functional analysis ({T}\"ubingen,
  1985)}, volume 122 of {North-Holland Math. Stud.}, pages 137--156.
  North-Holland, Amsterdam, 1986.

\bibitem{Makarow2000}
B.~M. Makarow and M.~R. Weber.  
\newblock {On the Asymptotic Behavior of some Positive Semigroups}.
\newblock {\em Preprint MATH-AN-09-2000, TU Dresden}, (2000), 20 pp. Available online from arxiv.org/abs/1901.04382v1 (2019).

\bibitem{Mischler2016}
S.~Mischler and J.~Scher.
\newblock {\em Spectral analysis of semigroups and growth-fragmentation equations.}
\newblock Ann. Inst. H. Poincar\'e Anal. Non Lin\'eaire 33 (2016), 849--898.

\bibitem{Murphy1990}
Gerard~J. Murphy.
\newblock {\em {$C^*$}-algebras and operator theory}.
\newblock Academic Press, Inc., Boston, MA, 1990.

\bibitem{Pedersen1979}
Gert K. Pedersen.
\newblock {\em $C^*$-algebras and their automorphism groups}.
\newblock Academic Press, Inc., London-New York, 1979.
\newblock London Mathematical Society Monographs, Volume 14.

\bibitem{Pichor2000}
Katarzyna Pich{\'o}r and Ryszard Rudnicki.
\newblock {\em Continuous {M}arkov semigroups and stability of transport equations.}
\newblock J. Math. Anal. Appl. 249 (2000), 668--685.

\bibitem{Ponce2018}
Augusto~C. Ponce and Daniel Spector.
\newblock A decomposition by non-negative functions in the {S}obolev space $W^{k,1}$.
  \newblock {\em Indiana Univ. Math. J.}, 69, 151--169, 2020.

\bibitem{Schaefer1960}
Helmut Schaefer.
\newblock {\em Halbgeordnete lokalkonvexe {V}ektorr\"aume. {III}.}
\newblock Math. Ann. 141 (1960), 113--142.

\bibitem{Schaefer1971}
Helmut~H. Schaefer.
\newblock {\em Topological vector spaces}.
\newblock Springer-Verlag, New York-Berlin, 1971.
\newblock Third printing corrected, Graduate Texts in Mathematics, Vol. 3.

\bibitem{Schaefer1974}
Helmut~H. Schaefer.
\newblock {\em Banach lattices and positive operators}.
\newblock Springer-Verlag, New York-Heidelberg, 1974.
\newblock Die Grundlehren der mathematischen Wissenschaften, Band 215.

\bibitem{Sherman1951}
S.~Sherman.
\newblock {\em Order in operator algebras.}
\newblock Amer. J. Math. 73 (1951), 227--232.

\bibitem{Stecenko1966}
 V.Ja.~Stecenko.
\newblock {\em Criteria of irreducibility of linear operators.}
\newblock {Uspekhi Math. Nauk}, v.21, no.5(131) (1966), 265--267.

\bibitem{Veitsblit1985a}
A.~I. Ve{\u\i}tsblit.
\newblock {\em A property of the boundary spectrum of nonnegative operators.}
\newblock Ukrain. Mat. Zh. 37 (1985), 114--116, 136.

\bibitem{Veitsblit1985}
A.~I. Ve{\u\i}tsblit and Yu.~I. Lyubich.
\newblock {\em The boundary spectrum of nonnegative operators.}
\newblock Sibirsk. Mat. Zh. 26(6) (1985), 24--28, 188.

\bibitem{Wulich2017}
Boris~Zacharowitsch Wulich.
\newblock {\em {Geometrie der Kegel in normierten R\"aumen.}}
\newblock De Gruyter, Berlin/Boston. 2017.
\newblock Transl.\ from Russian and edited by M.\ R.\ Weber.

\bibitem{ZabrKrasStec1967}
 P.P.~Zabrejko, M.A.~Krasnosel'skij, V.Ja.~Stecenko. 
\newblock {\em On estimates for the spectral radius of positive operators} (Russian).
\newblock {Math. Zametki} v.1, no.4 (1967), 461--470.

\end{thebibliography}
\end{document}